\documentclass{amsart}

\usepackage{amssymb, latexsym, epsfig}
\usepackage{blkarray}
\usepackage{graphicx}
\usepackage{amsmath}
\usepackage{amscd}
\usepackage{amsthm}
\usepackage{amsfonts}
\usepackage{amssymb}
\usepackage{color}
\usepackage{subfigure}
\usepackage[all]{xy}
\usepackage{pb-diagram}
\usepackage{blkarray}
\usepackage{multirow}
\usepackage{setspace}
\usepackage{enumitem}

\newtheorem{proposition}{Proposition}[section]
\newtheorem{thm}[proposition]{Theorem}
\newtheorem{corollary}[proposition]{Corollary}
\newtheorem{lem}[proposition]{Lemma}

\theoremstyle{definition}
\newtheorem{definition}[proposition]{Definition}

\usepackage[pagewise]{lineno}

\newtheorem*{rep@theorem}{\rep@title}
\newcommand{\newreptheorem}[2]{%
\newenvironment{rep#1}[1]{%
\def\rep@title{#2 \ref{##1}}%
\begin{rep@theorem}}%
{\end{rep@theorem}}}

\newreptheorem{theorem}{\bf{Theorem}}
\newreptheorem{lemma}{\bf{Lemma}}
\newreptheorem{corollary}{\bf{Corollary}}
\newreptheorem{proposition}{\bf{Proposition}}

\newcommand{\B}{{B}}
\newcommand{\el}{\mathcal{L}}
\newcommand{\LO}{\mathbb{L}_0}

\newcommand{\Z}{\mathbb{Z}}
\newcommand{\C}{\mathcal{C}}
\newcommand{\x}{\times}
\newcommand{\link}{L=K_1~\cup~\dots~\cup~K_m}
\newcommand{\linkprime}{L'=K_1'~\cup~\dots~\cup~K_m'}
\newcommand{\D}{\mathbb{D}}

\newcommand{\I}{[0,1]}
\newcommand{\bd}{\partial}
\newcommand{\F}{\mathcal{F}}
\newcommand{\N}{\mathbb{N}}
\newcommand{\Hom}{\text{Hom}}
\newcommand{\lk}{\operatorname{lk}}
\newcommand{\Arf}{\operatorname{Arf}}

\pdfoutput=1
\begin{document}

\title{Classification of links up to 0-solvability}

\author{Taylor E. Martin}
\address{Department of Mathematics\\
Sam Houston State University\
Huntsville, TX 77340 }
\email{taylor.martin@shsu.edu}
\thanks{}
\date{}

\begin{abstract}
The $n$-solvable filtration of the $m$-component smooth (string) link concordance group, as defined by Cochran, Orr, and Teichner, is a tool for studying smooth knot and link concordance that yields important results in low-dimensional topology. The focus of this paper is to characterize the set of 0-solvable links. We introduce a new equivalence relation on links called 0-solve equivalence and establish both an algebraic and a geometric classification of $\mathbb{L}_0^m$, the set of links up to 0-solve equivalence. We show that $\LO^m$ has a group structure isomorphic to the quotient $\F_{-0.5}^m/\F_0^m$ of concordance classes of string links and classify this group, showing that $$\LO^m \cong \F_{-0.5}^m/\F_0^m \cong \Z_2^m \oplus \Z^{m \choose 3} \oplus \Z_2^{m \choose 2}.$$ Finally, using results of Conant, Schneiderman, and Teichner, we show that 0-solvable links are precisely the links that bound class 2 gropes and support order 2 Whitney towers in the 4-ball.

\end{abstract}

\maketitle

\section{Introduction}\label{S:Intro}

Many open problems in low-dimensional topology center around the study of 4-manifolds. In the smooth category, even compact simply-connected 4-manifolds are unclassified. In the 1950's, Fox and Milnor introduced the notion of \emph{link concordance} a 4-dimensional equivalence relation; studying link concordance can contribute greatly to the understanding of 4-manifolds. In 1966, Fox and Milnor showed that concordance classes of knots form an abelian group called the \emph{knot concordance group}, $\C$. The knot concordance group has been well-studied since its introduction, but its structure is complicated and remains largely unknown. Here, we study the \emph{(string) link concordance group}, $\C^m$, where $m$ is the number of link components; when $m=1$, this is the knot concordance group. 

Cochran, Orr, and Teichner introduced the notion of \emph{n-solvability} of ordered, oriented links in 2003 \cite{COT}. This definition can be extended to string links; a string link is $n$-solvable if and only if its closure is an $n$-solvable link. Harvey gives a precise definition in \cite{H}. Cochran, Orr, and Tehichner's \emph{n-solvable filtration} $\F_n^m, n \in \frac{1}{2}\N$, is an infinite sequence of nested subgroups of $\C^m$ that can be thought of as an algebraic approximation to a link being slice \cite{COT}.

While the $n$-solvable filtration of $\C^m$ has been studied since its inception, many of the existing results discuss quotients of the filtration. Harvey shows that the quotients ${\F_n^m}/{\F_{n+1}^m}$ contain an infinitely generated subgroup \cite{H}. Cochran and Harvey showed that the quotients ${\F_n^m}/{\F_{n.5}^m}$ contain an infinitely generated subgroup \cite{CH}. Cochran, Harvey, and Leidy \cite{CHL} and Cha \cite{Cha} have made notable contributions in this area. In their seminal work on the $n$-solvable filtration, Cochran, Orr, and Teichner classified 0-solvable knots by showing that a knot is 0-solvable if and only if it has Arf invariant zero. Kauffman showed a knot has Arf invariant zero if and only if it is band-pass equivalent to the unknot \cite{Kauff}. As a corollary of the main result of this paper, we provide a parallel characterization of when a link is 0-solvable.

Let $\link$ be an ordered, oriented, $m$-component link such that the pairwise linking numbers $\lk{(K_i,K_j)}$ vanish. We call this set of links $\el^m$. We will establish an equivalence relation $\sim_0$ on $\el^m$ called 0-solve equivalence and classify the set $\LO^m := \el^m / \sim_0$ of links up to 0-solve equivalence both algebraically, using Milnor's invariants $\bar{\mu}(ijk)$ and $\bar{\mu}(iijj)$ as well as the $\Z_2$-valued Arf invariant of individual components, and geometrically, using the band-pass move, pictured in Figure \ref{fig:BP}. We establish a relationship among these conditions.

\begin{reptheorem}{thm:Main}
For two ordered, oriented $m$-component links $\link$ and $\linkprime$ with vanishing pairwise linking numbers, the following conditions are equivalent:
\begin{enumerate}
\item $L$ and $L'$ are 0-solve equivalent,
\item $L$ and $L'$ are band-pass equivalent,
\item $\Arf(K_i) = \Arf(K_i')$ \\
$\bar{\mu}_L(ijk) = \bar{\mu}_{L'}(ijk)$\\
$\bar{\mu}_L(iijj) \equiv \bar{\mu}_{L'}(iijj)$ $(\bmod 2)$ for all $i,j,k \in \{1, \dots, m\}$.
\end{enumerate}
\end{reptheorem}

As a corollary, we characterize when a link is 0-solvable.

\begin{repcorollary}{cor:0solvable}
For an ordered, oriented $m$-component link $\link$, with vanishing pairwise linking numbers, the following conditions are equivalent:
\begin{enumerate}
\item $L$ is 0-solvable,
\item $L$ is 0-solve equivalent to the $m$-component unlink,
\item $L$ is band-pass equivalent to the $m$-component unlink,
\item $\Arf(K_i) = 0$ \\
$\bar{\mu}_L(ijk) = 0 $\\
$\bar{\mu}_L(iijj) \equiv 0$ $(\bmod 2)$.\\
\end{enumerate}
\end{repcorollary}

We then classify the set $\LO^m$ of links up to 0-solve equivalence by giving an algorithm for choosing representatives of each 0-solve equivalence class of links and show that under band sum, $\LO^m$ has a group structure. 

\begin{repcorollary}{for:classify}
The set $\LO^m = \el^m / \sim_0$ forms an abelian group under band sum such that, for each $m$, the natural map $\frac{\F_{-0.5}^m}{\F_0^m} \rightarrow \LO^m$ induces an isomorphism, and

$$ \LO^m \cong \frac{\F_{-0.5}^m}{\F_0^m} \cong \Z_2^m \oplus \Z^{m \choose 3} \oplus \Z_2^{m \choose 2}$$
\end{repcorollary}

The work of Conant, Schneiderman, and Teichner \cite{Survey} broadens the scope of Theorem \ref{thm:Main} to include applications to the study of gropes and Whitney towers, which are geometric objects that are used in the study of 4-manifolds and are defined in section 5. 

\begin{repcorollary}{cor:CST}
For an ordered, oriented, $m$-component link $L$, the following are equivalent.
\begin{enumerate}
\item $L$ is 0-solvable.
\item $L$ bounds disjoint, properly embedded gropes of class 2 in $B^4$.
\item $L$ bounds properly immersed disks admitting an order 2 Whitney tower in $B^4$.
\end{enumerate}
\end{repcorollary}

\section{Preliminaries}

An $m$-component \emph{link} is an embedding $f: \bigsqcup_{i=1}^m S^1 \rightarrow S^3$ of $m$ disjoint circles into the $3$-sphere. We consider links up to smooth isotopy and denote an ordered link by $\link$, where each link component is oriented. Two ordered, oriented $m$-component links, $\link$ and $\linkprime$, are \emph{concordant} if they cobound $m$ disjoint, smoothly embedded annuli in $S^3 \x \I$ as in figure \ref{fig:conc}. Links which are concordant to the unlink are called \emph{slice} links.

\begin{figure}[h]
\centering
\includegraphics[scale=.3]{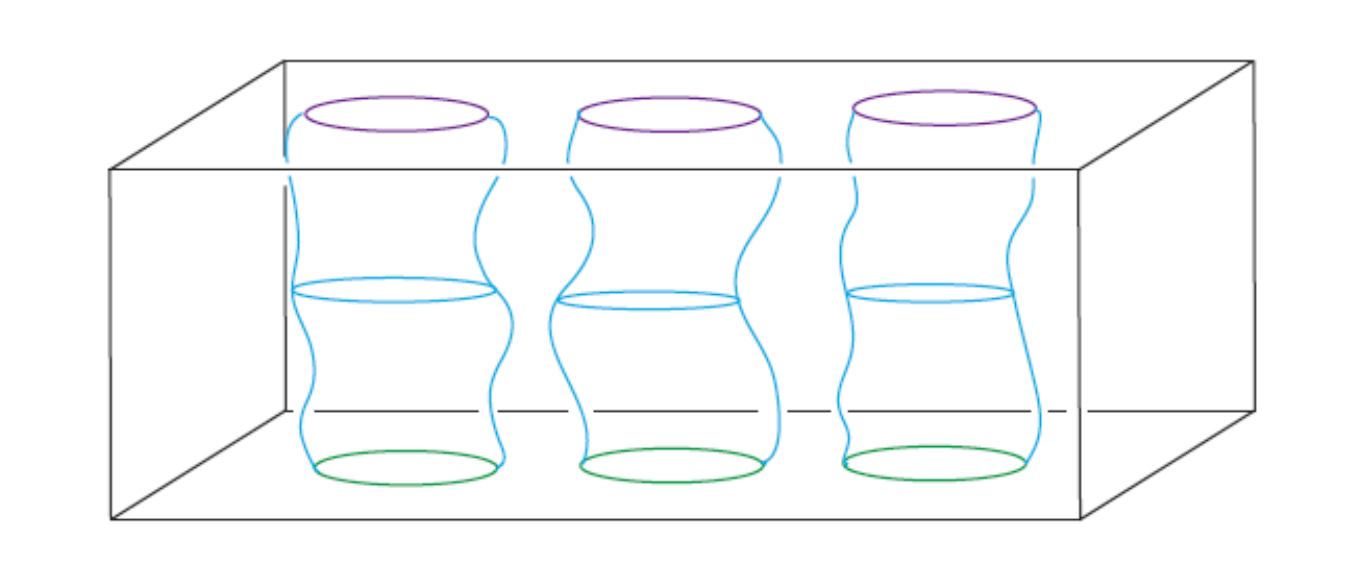}
\put(-235,40){$S^3 \x [0,1]$}
\put(-10,10){$L \subset S^3 \x \{0\}$}
\put(-10,70){$L' \subset S^3 \x \{1\}$}
\caption{Link Concordance}
\label{fig:conc}
\end{figure}

The concordance classes of knots form an abelian group under the operation of connected sum, but connected sum is not well-defined for links. However, this problem can be avoided by considering concordance classes of $m$-component \emph{string links}, which do form a group under the operation of stacking. An $m$-component (pure) $n$-\emph{string link} $D$, as defined by LeDimet in \cite{Dimet}, can be viewed as a generalization of an $m$-strand pure braid where we allow the strands to knot. Given any $m$-component string link $D$, we form the \emph{closure} of $D$, denoted $\hat{D}$, by gluing the standard $m$-component trivial string link $\{p_i\}_{i=1}^m \x I$ to $D$ along its boundary as pictured in figure \ref{fig:SLclosure}. A link in $S^3$ can be made into a string link by cutting along a carefully chosen ball \cite{HL1}.

\begin{figure}[h]
\centering
\includegraphics[scale=.2]{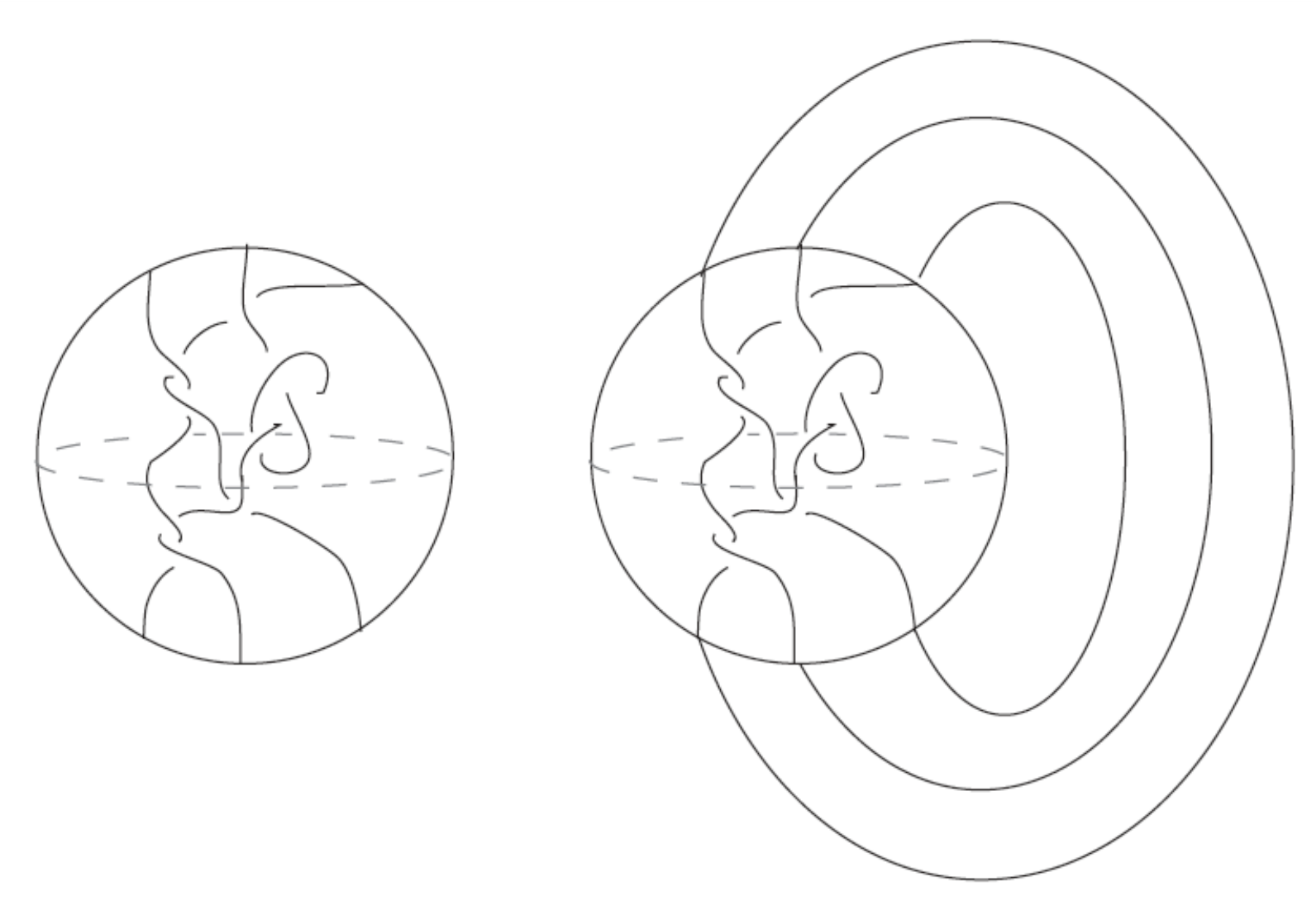}
\put(-165,50){$D$}
\put(5,50){$\hat{D}$}
\caption{A $3$-component string link $D$ and its closure $\hat{D}$}
\label{fig:SLclosure}
\end{figure} 

String links $D$ and $D'$ are \emph{concordant} if there exists a proper smooth submanifold $C$ of $\D^{3} \x \I$ such that $C$ is homeomorphic to $m$ disjoint copies of $\D^1 \x \I$, where $C \cap (\D^3 \x \{0\}) = D$, $C \cap (\D^3 \x \{1\}) = D'$, and $C \cap (S^2 \x \I) = U^0 \x \I$. Using this definition, we see that there is a natural extension of $C$ that gives a link concordance between the closures $\hat{D}$ and $\hat{D'}$. See \cite{Otto} for detailed definitions of string link concordance.

The set of concordance classes of $m$-component string links under \emph{stacking} form a group $\C^m$, called the \emph{string link concordance group} \cite{Dimet}. In the case of $m=1$, this is the knot concordance group. For $m \ge 2$, $\C^m$ is not abelian \cite{Dimet}. One tool to study the structure of this group, defined by Cochran, Orr, and Teichner \cite{COT}, is the \emph{n-solvable filtration}, $\{\F_n^m\}$: $$\{0\} \subset \dots \subset \F_{n+1}^m \subset \F_{n.5}^m \subset \F_n^m \subset \dots \subset \F_1^m \subset \F_{0.5}^m \subset \F_0^m \subset \C^m.$$ For $k \in \frac{1}{2}\N$, $\F_k^m$ is the collection of $k$-\emph{solvable} $m$-component links; the $n$-solvable filtration can be thought to ``approximate" sliceness as an $m$-component slice link is $n$-solvable for any $n$. We will focus on the first subgroup of this filtration, $\F_0^m$, the set of $0$-solvable $m$-component string links, and its counterpart for links, $\el_0^m$, the set of $m$-component, 0-solvable links. We give the definition of $n$-solvability as presented in \cite{Otto}.

\begin{definition}
An $m$-component link $\link$ is \emph{$n$-solvable} if the manifold $M_L$ obtained from performing 0-framed surgery on $L$ in $S^3$ bounds a compact, smooth $4$-manifold $W$ under the following conditions:

\begin{enumerate}
\item $H_1(M_L) \cong \Z^m$, and the map induced by inclusion, $i_*: H_1(M_L) \rightarrow H_1(W)$ is an isomorphism on the first homology.
\item $H_2(W)$ has a basis consisting of compact, connected, embedded, oriented surfaces $\{X_k,Y_k\}_{k=1}^r$ with trivial normal bundles, such that $X_k$ intersects $Y_k$ transversely, exactly once, with positive sign, and otherwise, the surfaces are disjoint. 
\item For each of the basis elements of $H_2(W)$, $\pi_1(X_k) \subset \pi_1(W)^{(n)}$ and $\pi_1(Y_k) \subset \pi_1(W)^{(n)}$, where $\pi_1(W)^{(n)}$ refers to the $n^{th}$ term of the derived series of the fundamental group.
\end{enumerate}
The 4-manifold $W$ is called an \emph{$n$-solution} for $L$. A link $L$ is $n.5$-solvable if it is $n$-solvable and for each $i$, $\pi_1(X_k) \subset \pi_1(W)^{(n+1)}$. A string link is $n$-solvable if and only if its closure is a $n$-solvable link. 
\end{definition}

We will use the notation $Q_W$ to denote the intersection form on the second homology of a 4-manifold $W$, and will will use $S \cdot T$ to indicate the signed count of intersections of embedded, oriented surfaces $S$ and $T$, possibly with boundary, in a $4$-manifold.

For a link $L$ to be 0-solvable, $H_1(M_L) \cong \Z^m$. This is true if and only if the pairwise linking numbers between components of $L$ vanish. Therefore, we will only consider links $\link$ such that $\lk(K_i, K_j)=0$ for $1 \le i < j \le m$. For convenience, we will denote this set of links by $\el^m$, and we will denote the collection of concordance classes of $m$-component string links with vanishing pairwise linking numbers as $\F_{-0.5}^m \subset \C^m$. 

In the 1950's, Milnor defined a classical family of link invariants called $\bar{\mu}$-invariants, denoted $\bar{\mu}_L(I)$, where $I=i_1i_2\dots i_k$ is a word of length $k$ and $i_j \in \{1, \dots, m \}$ refers to the $j^{th}$ link component of $L$ \cite{Mil1}, \cite{Mil2}. The integer $k$ is called the \emph{length} of the Milnor invariant. Milnor's invariants have some indeterminacy resulting from the choice of meridians of the link and so are only well-defined modulo the greatest common divisor of shorter length Milnor's invariants. Habegger and Lin show that this indeterminacy corresponds exactly to the choice of ways of representing a link as the closure of a string link \cite{HL1}. Milnor's invariants are concordance invariants \cite{Casson}. We will use $\bar{\mu}_L(iijj)$ and $\bar{\mu}_L(ijk)$, both of which are always well defined on the set $\el^m$ and can be thought of as higher order cup products. We give their definitions geometrically. Note that $i,j,$ and $k$ must be distinct.

The Milnor's invariants $\bar{\mu}_L(iijj)$ for a link $\link$ in $\el^m$ are also known as Sato-Levine invariants. We can compute $\bar{\mu}_L(iijj)$ by considering oriented Seifert surfaces $\Sigma_i$ and $\Sigma_j$ for $K_i$ and $K_j$ in the link exterior $S^3 - N(L)$. We may choose these surfaces in such a way that $\Sigma_i \cap \Sigma_j = \gamma \cong S^1$ \cite{Tim}. Then, we push the curve $\gamma$ off of one of the surfaces $\Sigma_i$ in the positive normal direction to obtain a new curve $\gamma^+$. The Sato-Levine invariant is $\lk(\gamma, \gamma^+)$; see figure \ref{fig:iijj} and is well-defined when the pairwise linking numbers of $L$ all vanish \cite{Tim}.

\begin{figure}[ht!]
\centering
\includegraphics[scale=.4]{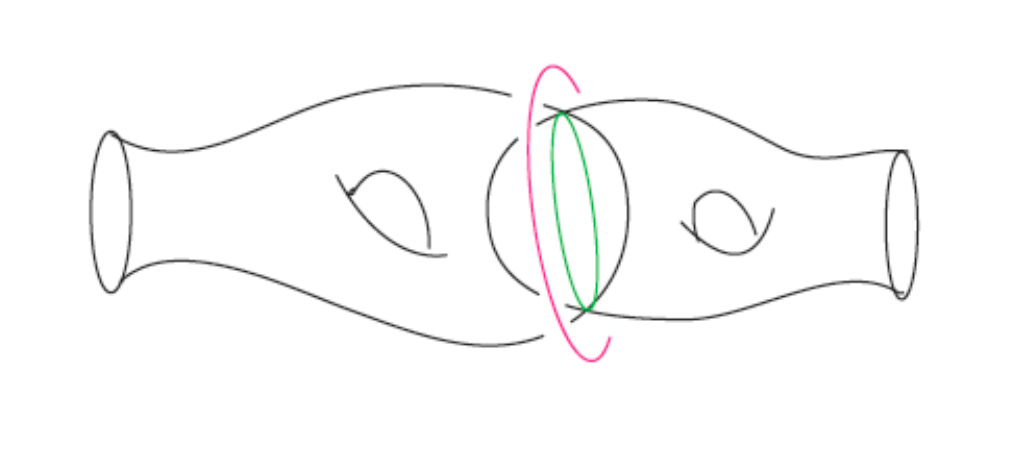}
\put(-160,70){$\Sigma_i$}
\put(-40,70){$\Sigma_j$}
\put(-100,85){$\gamma^+$}
\put(-80,80){$\gamma$}
\caption{Computing $\bar{\mu}_L(iijj)$}
\label{fig:iijj}
\end{figure}
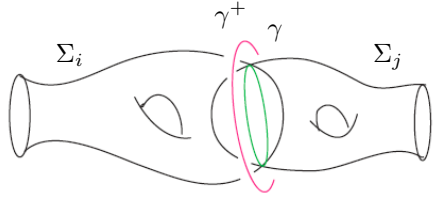

The Milnor's invariants $\bar{\mu}_L(ijk)$ for a link $\link$ also have a geometric definition \cite{Tim}. Let $\Sigma_i, \Sigma_j$, and $\Sigma_k$ be oriented Seifert surfaces for $K_i, K_j$, and $K_k$ in $S^3 - N(L)$. The intersection $\Sigma_i \cap \Sigma_j \cap \Sigma_k$ is a collection of points which are given an orientation induced by the outward normal on each Seifert surface. The count of these points up to sign gives us $\bar{\mu}_L(ijk)$ \cite{Tim}. Figure \ref{fig:ijk} depicts this.

\begin{figure}[ht!]
\centering
\includegraphics[scale=.4]{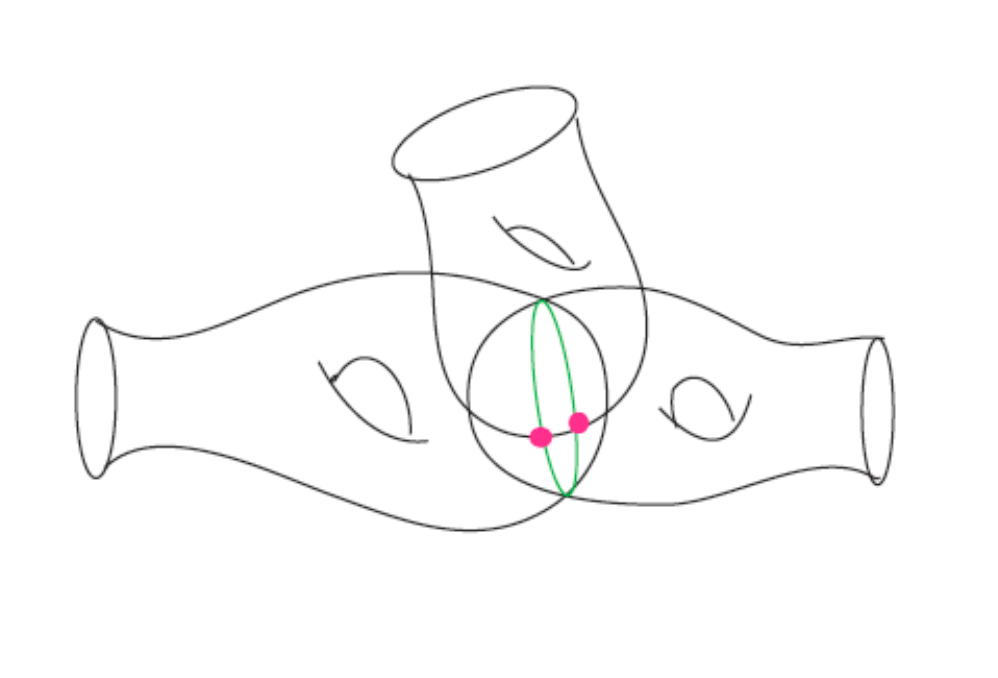}
\put(-160,30){$\Sigma_i$}
\put(-40,25){$\Sigma_j$}
\put(-70,90){$\Sigma_k$}
\caption{Computing $\bar{\mu}_L(ijk)$}
\label{fig:ijk}
\end{figure}

\begin{definition}
A \emph{band-pass move} on a link $\link$ is the local move pictured in figure \ref{fig:BP}. We require that both strands of each band belong to the same link component. Links $L$ and $L'$ are \emph{band-pass equivalent} if $L$ can be transformed into $L'$ through a finite sequence of band-pass moves and isotopy.

\begin{figure}[ht!]
\centering
\includegraphics[scale=.3]{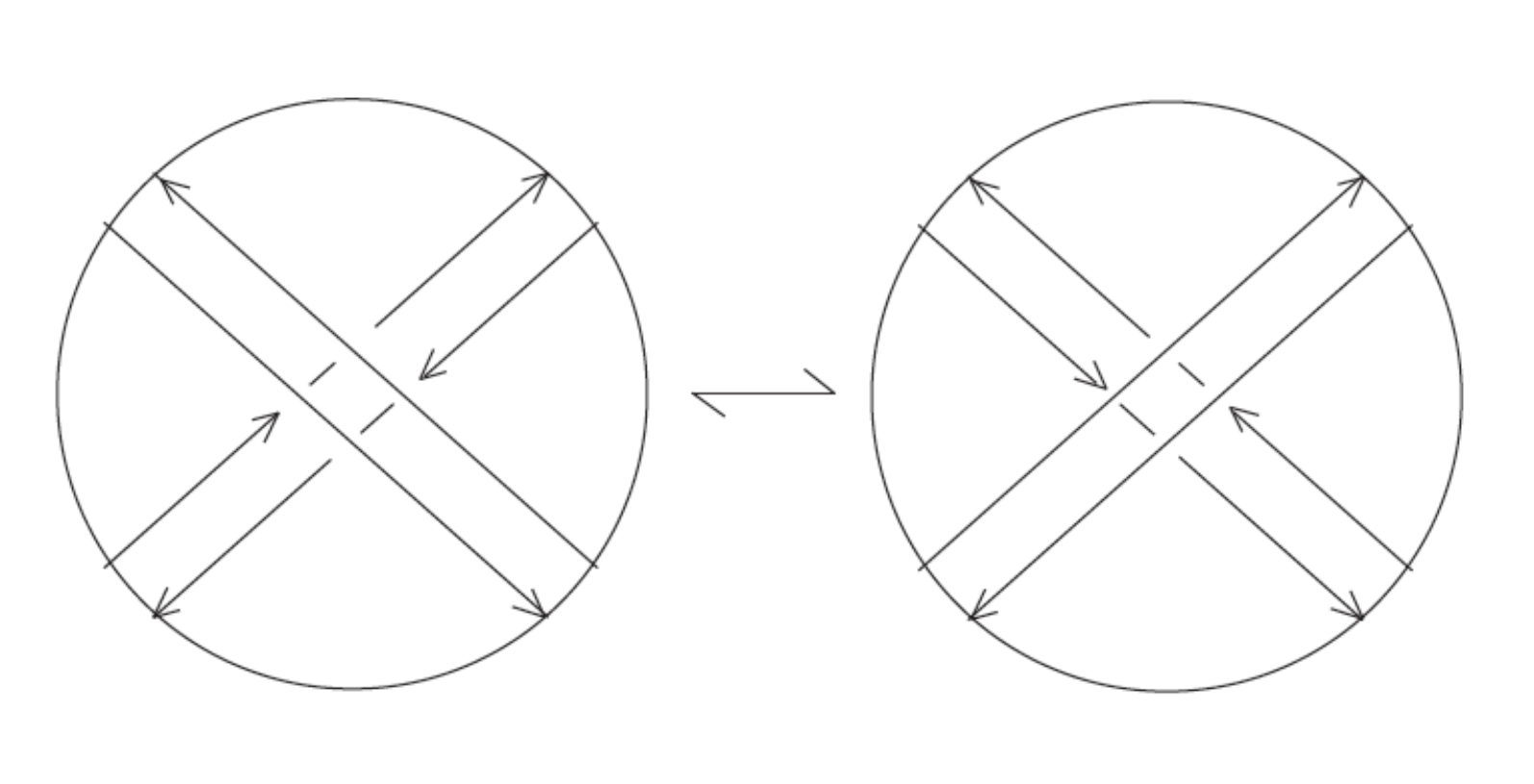}
\put(-220,70){$K_i$}
\put(-150,70){$K_j$}
\put(-98,70){$K_i'$}
\put(-27,70){$K_j'$}
\caption{A band-pass move}
\label{fig:BP}
\end{figure}
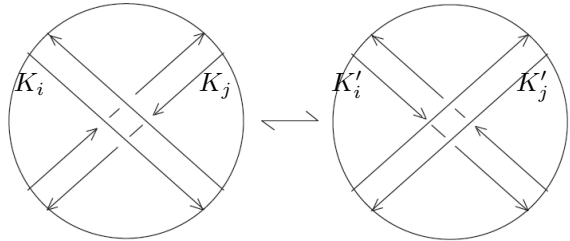
\end{definition}

\begin{definition}
A \emph{clasp-pass move} on a link $L$ is the local move in figure \ref{fig:CP}. Links $L$ and $L'$ are \emph{clasp-pass equivalent} if $L$ can be transformed into $L'$ through a finite sequence of clasp-pass moves and isotopy. Note that a clasp-pass move is a band-pass move. 
\begin{figure}[ht!]
\centering
\includegraphics[scale=.35]{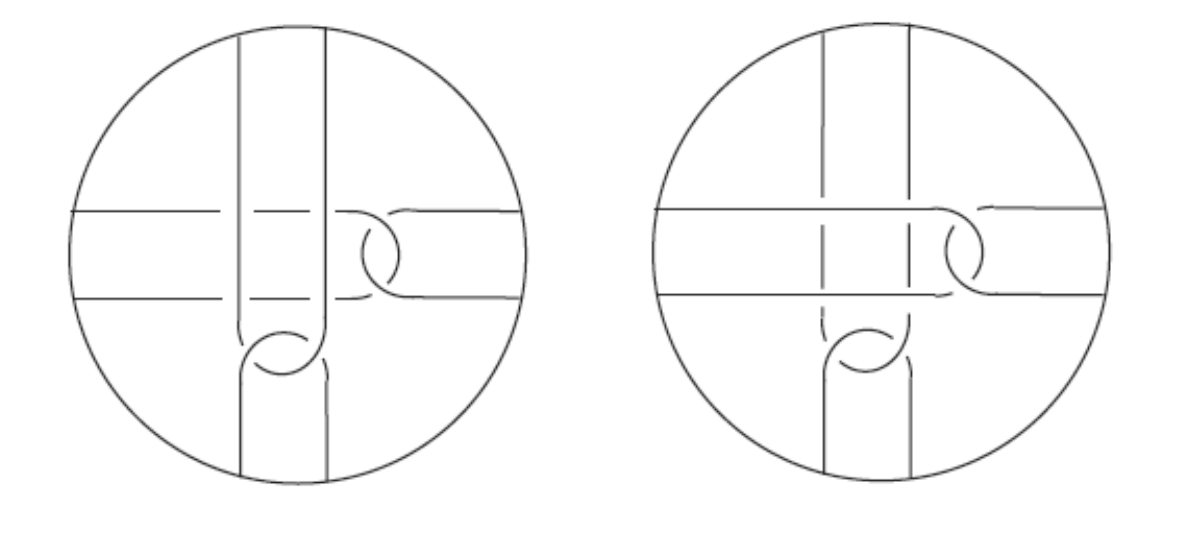}
\caption{A clasp-pass move}
\label{fig:CP}
\end{figure}
\end{definition}

\section{0-Solve Equivalence}

For an orientable 4-manifold $W^4$, a \emph{spin structure} on $W$ is a choice of trivialization of the tangent bundle $T_W$ over the 1-skeleton of $W$ that can be extended over the 2-skeleton. A manifold endowed with a spin structure is called a \emph{spin manifold}; $W$ is spin if and only if the second Stiefel-Whitney class $w_2(W) = 0$. For $W$ a smooth, closed 4-manifold such that $H_1(W)$ has no 2-torsion, $W$ is spin if and only if the intersection form $Q_W$ on $W$ is even (see \cite{Wild} section 4.3).

\begin{definition}
Suppose that $W$ is a 4-manifold such that $\bd W = M_L \bigsqcup -M_{L'}$, where $L$ and $L'$ are links. We define the \emph{closure} of $W$, which we denote $\hat{W}$, to be the closed 4-manifold that is obtained from $W$ by first attaching a 0-framed 2-handle to each meridinal curve in both $M_L$ and $M_{L'}$, as pictured in figure \ref{fig:lk}, and then attaching 4-handles to the $S^3$ and $-S^3$ boundary components. This forms the closed 4-manifold which we will call $\hat{W}$, as pictured in figure \ref{fig:closure}.

\begin{figure}[ht!]
\centering
\includegraphics[scale=.4]{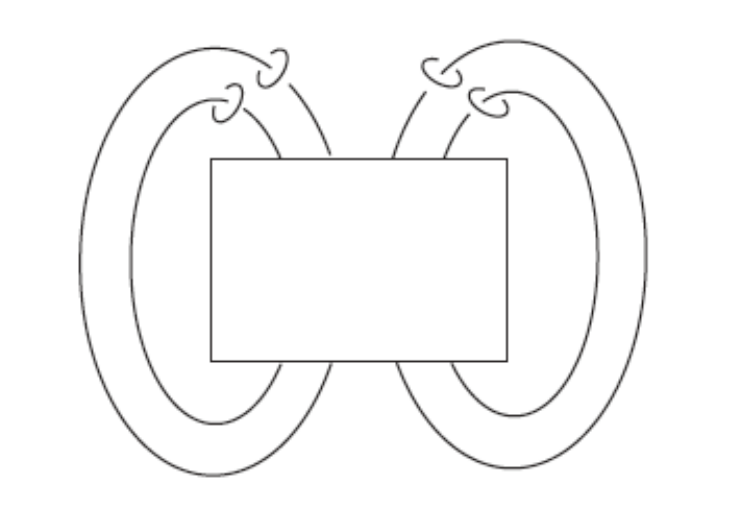}
\put(-77,50){L}
\put(-45,75){\tiny0}
\put(-65,95){\tiny0}
\put(-90,95){\tiny 0}
\put(-107,75){\tiny 0}
\put(-79,17){$\dots$}
\put(10,50){$= S^3$}
\caption{Closing off $W$ by attaching 2-handles along meridinal curves}
\label{fig:lk}
\end{figure}
 
\begin{figure}[ht!]
\centering
\includegraphics[scale=.4]{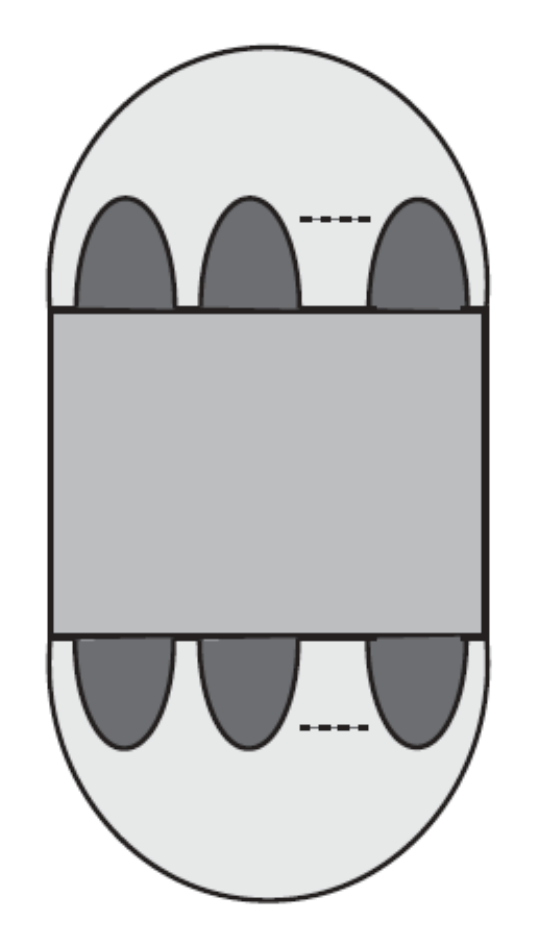}
\put(-60,90){$W$}
\put(-5,120){$M_L$}
\put(-8,60){$-M_{L'}$}
\caption{The closure of a 0-solve equivalence}
\label{fig:closure}
\end{figure}

\end{definition}

\begin{definition} \label{def:0se}
Let $L$ and $L'$ be ordered, oriented $m$-component links $\link$ and $\linkprime$ with vanishing pairwise linking numbers. We say $L$ is \emph{0-solve equivalent} to $L'$ if there exists a 4-manifold $W$ with $\bd W = M_L \bigsqcup -M_{L'}$ such that the following conditions hold:
\begin{enumerate}
\item The maps $i_*: H_1(M_L) \rightarrow H_1(W)$ and $j_*:H_1(-M_{L'}) \rightarrow H_1(W)$ induced by inclusion are isomorphisms such that $i_*(\mu_k) = j_*(\mu_k')$, where $\mu_k$ and $\mu_k'$ denote the meridians of the $k^{th}$ link components of $L$ and $L'$ respectively.
\item $\Z^{2r} \cong H_2(W, \bd W_-)$, where $\bd W_- = -M_{L'}$ and $\bd W_+ = M_L$, has a basis consisting of compact, oriented, embedded, connected pairs of surfaces $\{X_k,Y_k\}_{k=1}^r$ with trivial normal bundles such that the surfaces are disjoint, with the exception that, for each $k$, $X_k $ and $Y_k$ intersect transversely at exactly one point, oriented with positive direction. 
\item $\hat{W}$ is a spin 4-manifold. 
\end{enumerate}

This manifold $W$ is called a \emph{0-solve equivalence} between $L$ and $L'$. 
\end{definition}

Note that Cha gives a definition of $n$-solvable equivalence that is not equivalent to Definition \ref{def:0se} as it omits the spin condition. See \cite{Cha2} Remark 2.10.

Consider the long exact sequence of the pair $(W, \bd W_-)$ given below: 
$$\dots \rightarrow H_3(W, \bd W_-) \rightarrow H_2(\bd W_-) \stackrel{j_*}{\rightarrow} H_2(W) \twoheadrightarrow H_2(W,\bd W_-) \stackrel{0}{\rightarrow} H_1(\bd W_-) 
\stackrel{\cong}{\longrightarrow} H_1(W) \rightarrow \dots$$

Note that the map induced by inclusion, $j_*:H_1(\bd W_-) \rightarrow H_1(W)$ is an isomorphism and that $H_3(W, \bd W_-) \cong H^1(W, \bd W_+) \cong \Hom(H_1(W, \bd W_+), \Z) = 0$. Thus, $j_*: H_2(\bd W_-) \rightarrow H_2(W)$ is injective. We refer to $j_*(H_2(\bd W_-))$ simply by $H_2(\bd W_-)$. The short exact sequence obtained from the above long exact sequence splits, as $H_2(W, \bd W_-)$ is free. Thus, $H_2(W) = H_2(W, \bd W_-) \oplus j_*(H_2(\bd W_-))$ and for a 0-solve equivalence $W$, 

\begin{equation} \label{eq:1}
\frac{H_2(W)}{H_2(\bd W_-)} \cong H_2(W, \bd W_-). 
\end{equation}

\begin{proposition}\label{equivrel}
0-solve equivalence in an equivalence relation on $\el^m$.
\end{proposition}

\begin{proof}
\begin{enumerate}
\item To show reflexitivity, for an ordered, oriented link $L$ in $\el^m$, we consider the 4-manifold $W = M_L \x \I$, which we will show to be a 0-solve equivalence from $L$ to itself. The first two conditions of 0-solve equivalence are trivial. The generators of $H_2(\hat{W})$ are $\{\hat{\Sigma}_i, F_i\}_{i=1}^m$, where $\Sigma_i \subset S^3 - N(L) \subset M_L$ is a Seifert surface for the link component $K_i $, $\hat{\Sigma}_i$ is the closed surface $\Sigma_i \cup_{\ell_i} \D^2$ where $\ell_i$ is the 0-framing on $K_i$, and $F_i$ is the surface consisting of the cores of the $i^{th}$ added 2-handles attached to $\mu_i \x \I \subset M_L \x \I$. Each generator $\hat{\Sigma}_i$ and $F_i$ has self-intersection zero. Thus, the intersection form $Q_{\hat{W}}$ is even, and so $\hat{W}$ is spin.

\item To show symmetry, let $W$ be a 0-solve equivalence from link $L$ to $L'$ in $\el^m$. We reverse orientation on $W$ to obtain $W'$, where $\bd W'_- = -M_L$ and $\bd W'_+=M_{L'}$. We wish to show that $W'$ is a 0-solve equivalence from $L' $ to $L$. The first condition on the inclusion maps inducing isomorphisms on first homology is inherited from $W$, as $W$ and $W'$ differ only by orientation.

Then, since $H_1(W, \bd W_+)=0$,
\begin{align*}
H_2(W, \bd W_-) &\cong H^2(W, \bd W_+) \\
&  \cong \operatorname{Hom}(H_2(W, \bd W_+),\Z) \oplus \operatorname{Ext}(H_1(W, \bd W_+), \Z)\\
&\cong  H_2(W, \bd W_+) \\
\end{align*}

%
%





Also, $H_2(\bd W_+) \cong \Z^m$ is a free abelian group with basis $\{[\hat{\Sigma_i}]\}$ described as follows. We choose Seifert surfaces $\{\Sigma_i\}_{i=1}^m \subset S^3-L$ for link components $K_i$ with agreeing orientation. Let $\hat{\Sigma}_i$ be a closed surface in $\bd W_+$ defined by adding a disk to the preferred longitude of $K_i$. Similarly, $H_2(\bd W_-) \cong \Z^m$ is a free abelian group on basis $\{[ \Sigma'_i]\}_{i=1}^m$, where $\Sigma'_i \subset S^3-L'$ is a Seifert surface for link component $K_i'$ with agreeing orientation. Again, $\hat{\Sigma'_i}$ is a  closed surface in $\bd W_-$ defined by adding a disk to the preferred longitude of $K'_i$. Since the classes $[\hat{\Sigma_i}]$ and $[\hat{\Sigma'_i}]$ are carried by $\bd W$, the intersection form vanishes on their span in $H_1(W)$. Further, $\hat{\Sigma_i}$ and $\hat{\Sigma'_i}$ are disjoint from  $\{X_k,Y_k\}$, the basis elements of $H_2(W, \bd W_-)$ from Definition \ref{def:0se}. Since $i_*( [\hat{\Sigma_i}])=j_*([\hat{\Sigma_i']})$, the images of $i_*:H_2(\bd W_+) \rightarrow H_2(W)$ and $j_*:H_2(\bd W_-)\rightarrow H_2(W)$ are equal. From the following diagram $$H_2(W', \bd W'_-)  \stackrel{\cong}{\longleftarrow} \frac{H_2(W)}{i_*(H_2(\bd W_+))} = \frac{H_2(W)}{j_*(H_2(\bd W_-))} \stackrel{\cong}{\longrightarrow} H_2(W, \bd W_-).$$ 
since $\{[X_k],[Y_k]\}$ forms a basis for $H_2(W, \bd W_-)$, it also forms a basis for $H_2(W', \bd W'_-)$, as desired. Lastly, we observe that $\hat{W'}$ is $\hat{W}$ with opposite orientation. As $\hat{W}$ is spin, so is $\hat{W'}$. 

\item To show transitivity, let $W$ be a 0-solve equivalence from link $L$ to link $L'$ and let $W'$ be a 0-solve equivalence from link $L'$ to $L''$, where $L, L',$ and $L''$ are ordered, oriented $m$-component links with vanishing pairwise linking numbers. Let $V = W \cup_{M_{L'}} W'$. We wish to show that $V$ is a 0-solve equivalence from $L$ to $L''$. The first condition of 0-solve equivalence follows from the fact that $W$ and $W'$ are 0-solve equivalences; $H_1(\bd W_+) \cong H_1(W) \cong H_1(\bd W_-) = H_1(\bd W'_+) \cong H_1(W') \cong H_1(\bd W'_-)$, where each isomorphism is induced by inclusion. Thus, a Mayer-Vietoris argument shows that $H_1(V) \cong H_1(W)$ and $H_1(V) \cong H_1(W')$. The Mayer-Vietoris sequence for $V$ is:
$$\rightarrow H_3(V) \rightarrow H_2(W \cap W') \stackrel{(i_*,-j_*)}{\rightarrow} H_2(W) \oplus H_2(W') \twoheadrightarrow H_2(V) \stackrel{0}{\rightarrow} H_1(W \cap W') \rightarrow.$$

This tells us that $$H_2(V) \cong \frac{H_2(W) \oplus H_2(W')}{(i_*,-j_*)(H_2(W \cap W'))}.$$ Using the splittings $H_2(W) = H_2(W, \bd W_-) \oplus j_*(H_2(\bd W_-))$ and  $H_2(W') = H_2(W', \bd W'_+) \oplus j_*(H_2(\bd W'_+))$, we observe that the induced map $(i_*,-j_*):H_2(W \cap W') \rightarrow H_2(W) \oplus H_2(W')$ acts by $(i_*,-j_*)(a) = ((0,a),(0,-a))$. From this, we deduce that $H_2(V) \cong H_2(W, \bd W_-) \oplus H_2(W')$. 

Next, we consider the long exact sequence of the pair $(V, \bd V_-)$. 
$$\rightarrow H_3(V, \bd V_-) \rightarrow H_2(\bd V_-) \rightarrow H_2(V) \twoheadrightarrow H_2(V, \bd V_-) \stackrel{0}{\rightarrow} H_1(\bd V_-) \stackrel{\cong}{\rightarrow} H_1(V) \rightarrow.$$ This tells us that $H_2(V, \bd V_-) \cong \frac{H_2(V)}{i_*(H_2(\bd V_-))}$. Coupled with the splitting $H_2(W') = H_2(W', \bd W'_-) \oplus j_*(H_2(\bd W'_-))$ and the fact that $\bd V_- = \bd W'_-$, we have that 

\begin{align*}
H_2(V, \bd V_-) \cong & \frac{H_2(W, \bd W_-) \oplus H_2(W', \bd W'_-) \oplus H_2(\bd W'_-)}{i_*(H_2(\bd V_-))}\\
& \cong \frac{H_2(W, \bd W_-) \oplus H_2(W', \bd W'_-) \oplus H_2(\bd W'_-)}{i_*(H_2(\bd W'_-))}\\
& \cong H_2(W, \bd W_-) \oplus H_2(W', \bd W'_-).\\
\end{align*}

This isomorphism tells us that the rank of $H_2(V, \bd V_-)$ is $2r+2r'$, where $2r$ and $2r'$ are the ranks of $H_2(W, \bd W_-)$ and $H_2(W', \bd W'_-)$ respectively. Using the inclusion-induced isomorphisms given in equation \ref{eq:1}, we can consider the set $\{X_k, Y_k\}_{k=1}^r \cup \{X_k', Y_k'\}_{k=1}^{r'}$ as surfaces in $(V,\bd V_-)$ which are linearly independent and primitive because they have intersection duals. Because this collection has full rank, we see that $H_2(V, \bd V_-)$ has the desired basis in Definition \ref{def:0se}.


 Finally, we must show that $\hat{V}$ is a spin manifold. We show that the basis elements of $H_2(\hat{V})$ have even self-intersection. We know that the closures $\hat{W}$ and $\hat{W}'$ are spin manifolds. The generators of $H_2(\hat{W})$ are the surfaces $\{X_k, Y_k\}$ with self-intersection zero, capped Seifert surfaces for each link component, $\{\hat{\Sigma_k}\}$, and capped surfaces $\{\hat{S_k}\}$, where $S_k$ is an embedded surface in $W$ realizing the trivial homology class $[i_*(\mu_k)- j_*(\mu_k')]$. Similarly, the generators of $H_2(\hat{W'})$ are the surfaces $\{X_k', Y_k'\}$ with self-intersection zero, capped Seifert surfaces for each link component, $\{\hat{\Sigma_k'}\}$, with self-intersection zero, and capped surfaces $\{\hat{S_k'}\}$, where $S_k'$ is an embedded surface in $W'$ realizing the trivial homology class $[i_*(\mu_k')- j_*(\mu_k'')]$. A basis for $H_2(\hat{V})$ is the set  $\{X_k,Y_k\} \cup \{X_k',Y_k'\} \cup  \{\hat{\Sigma_k'}\} \cup \{\hat{T_k}\} $ where $\hat{T_k}$ is the capped surface $T_k = S_k \cup_{\mu_k'} S_k'$. Since the caps have no self-intersection due to the framing, the self intersection $\hat{T_k} \cdot \hat{T_k} = S_k \cdot S_k + S_k' \cdot S_k'$, a signed count of triple points. Because $\hat{W}$ and $\hat{W'}$ are spin and have no 2-torsion, $S_k \cdot S_k$ and $S_k' \cdot S_k'$ are both even. Therefore, $\hat{V}$ has even intersection form and is spin.

\end{enumerate}
\end{proof}

\begin{proposition}\label{prop:0solv}
If $L \sim_0 L'$ via a 0-solve equivalence $W$, and if $V$ is a 0-solution for $L'$, then $L$ is 0-solvable, and $V' = W \bigsqcup_{M_{L'}} V$ is a 0-solution for $L$. 
\end{proposition}

\begin{proof}

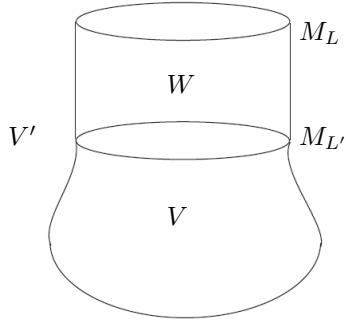
\begin{figure}[ht!]
\centering
\includegraphics[scale=.3]{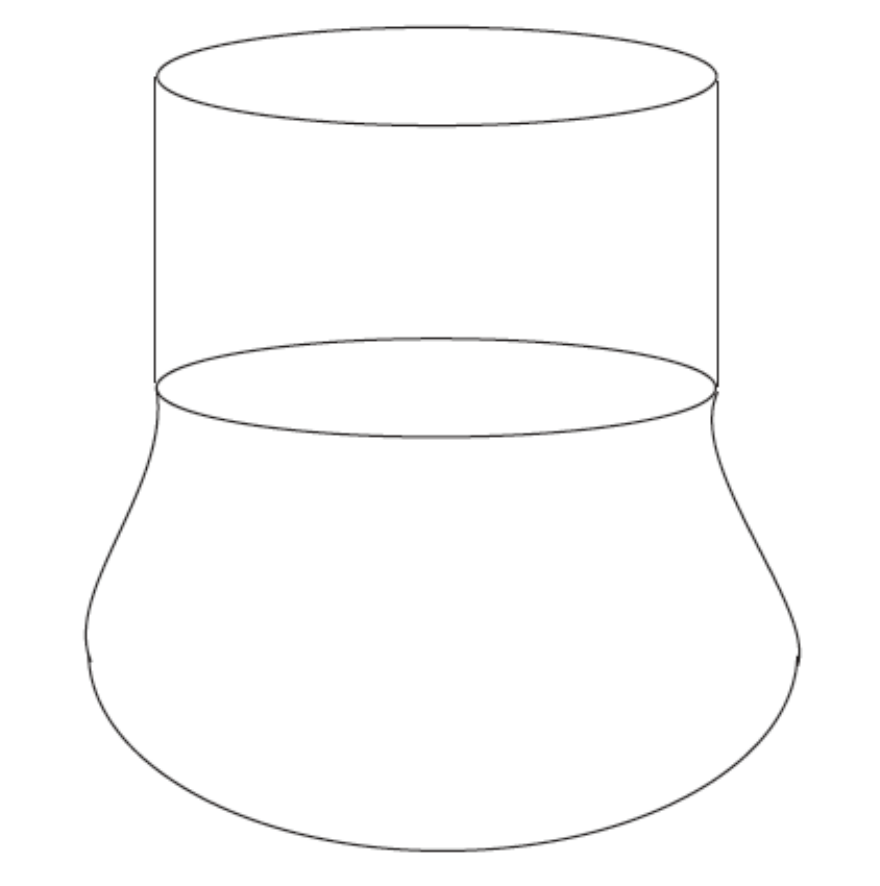}
\put(-70,40){$V$}
\put(-70,90){$W$}
\put(-130,70){$V'$}
\put(-20,110){$M_L$}
\put(-20,70){$M_{L'}$}
\caption{A 0-solution for $L$}
\end{figure}
The first condition of 0-solvability, that $i_*:H_1(M_L) \rightarrow H_1(V')$ is an isomorphism, follows from the fact that $W$ is a 0-solve equivalence and that $V$ is a 0-solution; $H_1(M_L) \cong H_1(W) \cong H_1(M_{L'}) \cong H_1(V)$ where each map is induced by inclusion. A Mayer-Vietoris computation shows that $H_1(V') \cong H_1(W)$.

We know that $H_2(V)$ has basis consisting of surfaces $\{X_k,Y_k\}_{k=1}^r $ that intersect pairwise once, transversely. We also know that $H_2(W, \bd W_-)$ has basis of surfaces $\{X_k',Y_k'\}_{k=1}^{r'}$ that intersect pairwise once transversely. We will show that $H_2(V')$ has basis $\{X_k,Y_k\} \cup \{X_k',Y_k'\}$. 
We consider the Mayer-Vietoris sequence for $V'$.
$$\rightarrow H_3(V') \rightarrow H_2(W \cap V) \stackrel{(i_*,-j_*)}{\rightarrow} H_2(W) \oplus H_2(V)  \rightarrow H_2(V') \stackrel{0}{\rightarrow} H_1(W \cap V) \rightarrow.$$ 

Similarly to in Proposition \ref{equivrel}, 
\begin{align*}
H_2(V') & \cong  \frac{H_2(W) \oplus H_2(V)}{(i_*,-j_*)(H_2(W \cap V))}\\
& \cong  \frac{H_2(W, \bd W_-) \oplus H_2(\bd W_-)\oplus H_2(V)} {(i_*,-j_*)(H_2(W \cap V))}\\
&  \cong  H_2(W, \bd W_-) \oplus H_2(V). 
\end{align*}
Note that this isomorphism shows that the rank of $H_2(V')$ is $2r+2r'$, where $2r$ and $2r'$ are the ranks of $H_2(V)$ and $H_2(W,\bd W_-)$, respectively. Using the inclusion induced maps and equation \ref{eq:1}, we consider the set $\{X_k,Y_k\}_{k=1}^r \cup \{X'_k,Y'_k\}_{k=1}^{r'}$. In $V'$, these surfaces are oriented and embedded, have trivial normal bundles, and are disjoint, with the exception that for each $k$, $X_k$ and $Y_k$ intersect each other exactly once and $X'_k$ and $Y'_k$ intersect each other exactly once. Since this collection has full rank for $H_2(V')$, we can say that it is a basis, and so $V'$ is a 0-solution for $L$.
\end{proof}

\begin{proposition}\label{prop:sublinks}
For $\link$ and $\linkprime$ two $m$-component 0-solve equivalent links with vanishing pairwise linking numbers, the corresponding $k$-component sublinks $J = K_{i_1} \cup \dots \cup K_{i_k}$ and $J' = K'_{i_1} \cup \dots \cup K'_{i_k}$, where $i_j = i_j' \in \{1, \dots , m \}$ are also 0-solve equivalent.
\end{proposition}

\begin{proof}
Note that, by induction and reordering both links simultaneously, it suffices to show that the sublinks obtained by removing the first component of $L$ and $L'$ are $0$-solve equivalent. Let $J = K_2 \cup \dots \cup K_m$ and $J' = K'_2 \cup \dots \cup K'_m$. 

Let $W$ be the $4$-manifold realizing the $0$-solve equivalence between $L$ and $L'$. We form a cobordism between $M_J$ and $M_{J'}$ by attaching two $0$-framed $2$-handles to $\partial W$ with attaching spheres $\mu_1$ and $\mu_1'$, the meridians of the link components $K_1$ and $K_1'$. This yields a new $4$-manifold, $V$. By performing handleslides, and since 0-framed surgery on the Hopf link is homeomorphic to $S^3$, we see that $\partial V = M_J \bigsqcup -M_{J'}$. We wish to show that $V$ is a $0$-solve equivalence between $J$ and $J'$. 

By definition, $i_*:H_1(M_L) \hookrightarrow H_1(W)$ and $j_*:H_1(M_{L'}) \hookrightarrow H_1(W)$ are both isomorphisms. Moreover, $H_1(M_L) \cong <\mu_1, \dots, \mu_m > \cong \mathbb Z^m$ and $H_1({M_L'}) \cong <\mu_1', \dots, \mu_m' > \cong \mathbb Z^m$. Considering the effect on homology of attaching a $2$-cell and the fact that $i_*(\mu_1) = j_*(\mu_1')$, we see that the maps induced by inclusion, $i_*: H_1(M_J) \rightarrow H_1(V)$ and $j_*: H_1(M_{J'}) \rightarrow H_1(V)$ are both isomorphisms, and it still holds that $i_*(\mu_k) = j_*(\mu_k)$ for $k \ne 1$. 

We now consider how $H_2(V)$ differs from $H_2(W)$. By attaching a 2-handle to $W$ along the curve $\mu_1$, we do not change second homology, as $<i_*(\mu_1)>$ is infinite in $H_1(W)$. Then, attaching a 2-handle along the curve $\mu_1'$ creates an infinite cyclic direct summand in $H_2(V)$ generated by $B$, the oriented surface in $V$ obtained by capping off the meridians $\mu_1$ and $\mu_1'$ in $W$ plus a surface in $W$ realizing the null-homology between $\mu_i$ and $\mu_i'$ as pictured in figure \ref{fig:B}.

\begin{figure}[h]
\centering
\includegraphics[scale=.3]{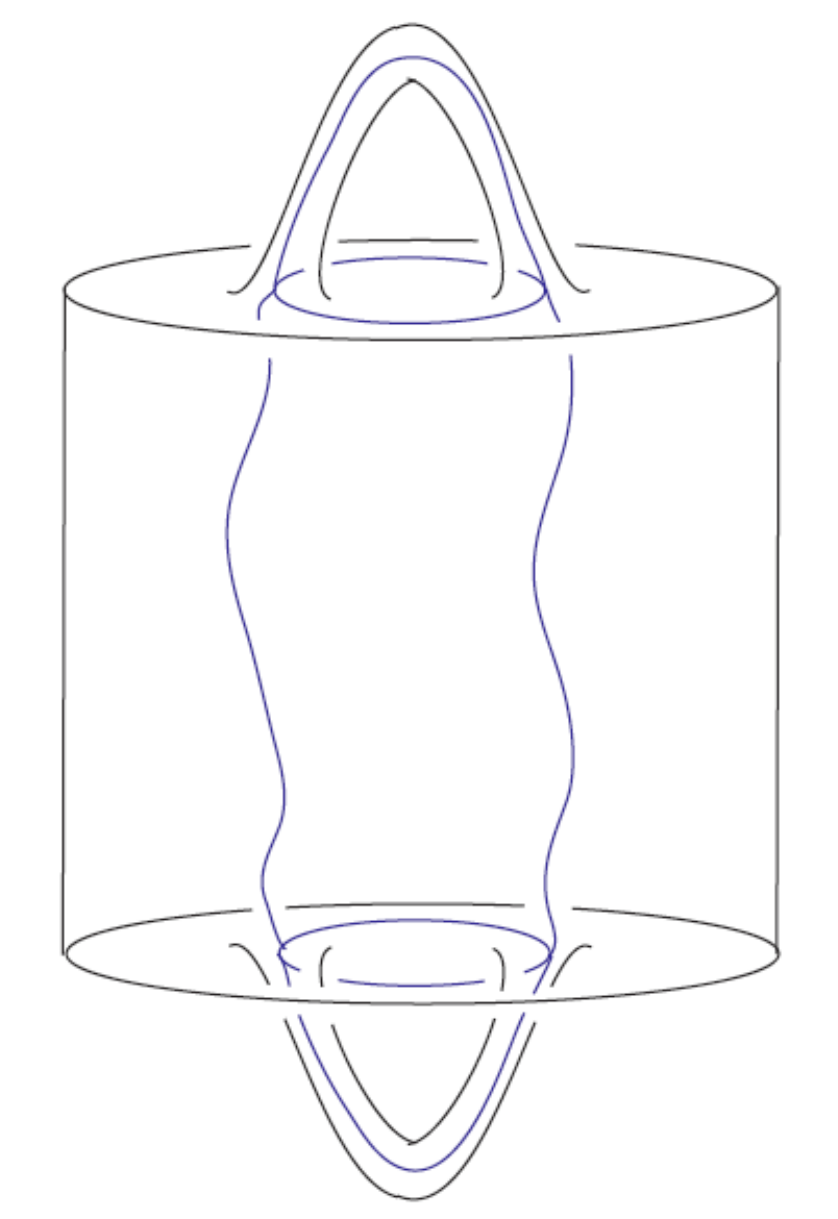}
\put(-30,110){$B$}
\put(-80,110){$\circlearrowleft$}
\caption{The surface $B \subset V$}
\label{fig:B}
\end{figure}

Next, we consider how $H_2(V, \bd V_-) \cong \frac{H_2(V)}{H_2(\bd V_-)}$ differs from $H_2(W, \bd W_-) \cong \frac{H_2(W)}{H_2(\bd W_-)}$. Let $\Sigma_1'$ be a Seifert surface for $K_1'$ in $S^3-N(L') \subset M_{L'}$, and let $S_1'$ be a closed oriented surface in $M_{L'}$ obtained by capping off $\Sigma_1'$ with the disk obtained from performing zero-framed surgery on $L'$. We notice that $S_1' \subset \bd W_-$  but $S_1'$ is not in $ \partial V_-$. From its construction, we can assume that $B$ is an annulus near $\mu_i$ and $\mu_i'$; therefore, we can assume that $B \cap M_{L'} = \mu_1'$, and so $B$ intersects $S_1'$ transversely at precisely one point. After possibly changing the orientation on $B$, we may assume that $B \cdot S_1' = +1$.  

From Poincare duality, we have an intersection form $H_2(V, \bd V_-) \x H_2(V, \bd V_+) \rightarrow \Z$ where $(S_1',B) \mapsto +1$. Therefore, the class $[S_1'] \in H_2(V, \bd V_-)$ is nontrivial, where $S_1'$ can be made disjoint from the generators of $H_2(W)$, and so $[S_1']$ is primitive in $H_2(V, \bd V_-)$. 

We have that $H_2(V, \bd V_-) \cong H_2(W, \bd W_-) \oplus \Z \oplus \Z$ where the extra homology is generated by $\{ [B], [S_1']\}$. We must show that $H_2(V, \bd V_-)$ has generators that intersect as in the definition of 0-solve equivalence. 

So far, we have established that the intersection matrix for $H_2(V, \bd V_-)$ looks like: 

\[
  \begin{blockarray}{ccc|cccc}
    & S_1' & B & L_1 & D_1 \dots & L_g & D_g  \\
    \begin{block}{c(cc|cccc@{\hspace*{5pt}})}
    S'_1 & 0 & 1& 0 & 0 \dots & 0&0 \\
    B & 1 & \star & \gamma_1 & \gamma_2 \dots & \gamma_{2g-1} & \gamma_{2g} \\
    \cline{1-7}
    L_1 & 0 & \gamma_1 &  \BAmulticolumn{4}{c|}{\multirow{4}{*}{$A$}}\\
   D_1 & 0 & \gamma_2  & &&&\\
    \vdots & \vdots & \vdots & &&&\\
  L_g & 0 &\gamma_{2g-1} & &&&\\
   D_g & 0 & \gamma_{2g} & & & & \\
    \end{block}
  \end{blockarray}
\]

Here, $A$ is the intersection matrix for $H_2(W, \bd W_-)$, $\{\gamma_i\}$ are integers representing the intersection of $B$ with the generators of $H_2(W, \bd W_-)$, and $\star$ is an integer representing the self-intersection $B \cdot B$. We first seek to show that, for some choice of $B$, we have $\star = 0$. 

Considering the way we constructed $V$, we observe that any self-intersection of $B$ must occur in the interior of $W$. Moreover, $B$ is a surface in $\hat{W}$, the closure of the 0-solve equivalence $W$. Because $\hat{W}$ is spin, we know that $B \cdot B \equiv 0 $ $(\bmod 2)$. We introduce a change of basis to find a surface $B_k$ that has trivial self-intersection. Suppose that $B \cdot B  = 2k$ where $k<0$. Then, let $S_1'^+$ be a push-off of the surface $S_1'$. Let $\alpha_1$ be an arc in $V$ from $B$ to $S_1'^+$ that does not intersect the other basis elements $\{X_i,Y_i\}$ of $H_2(V, \bd V_-)$. We define a new surface $B_1$ to be the result of performing ambient surgery along $\alpha_1$, a process referred to as ``tubing" as pictured in figure \ref{fig:Tubing}. In homology, $[B_1] = [B + S_1']$. Note that we can choose the arc so that the orientation on our new surface matches up with the original orientations on $B$ and $S_1'$. This new surface $B_1$ also intersects $S_1'$ exactly once transversely, and has self-intersection $B_1 \cdot B_1 = B \cdot B + S_1' \cdot S_1' + 2(B \cdot S_1') = B \cdot B +2$. Since $B \cdot B = 2k$, by repeating this process $k$ times, we will have a surface $B_k$ that intersects $S_1'$ exactly once and has self intersection zero. Note that, if $B \cdot B = 2k$ where $k >0$, we would let $-S_1'$ be the surface $S_1'$ with opposite orientation. We would then let $\alpha_1$ be an arc from $B$ to $-S_1'^+$ and we would then define $B_1$ as above. Under these conditions, the homology class $[B_1] = [B - S_1']$ and $B_1 \cdot B_1 = B \cdot B + 2(B \cdot -S_1') = B \cdot B - 2$. Therefore, we have that $[B_k] = [B \pm kS_1']$, and we use this change of basis for $H_2(V, \bd V_-)$, replacing $B$ with $B_k$. 

\begin{figure}[ht!]
\centering
\includegraphics[scale=.35]{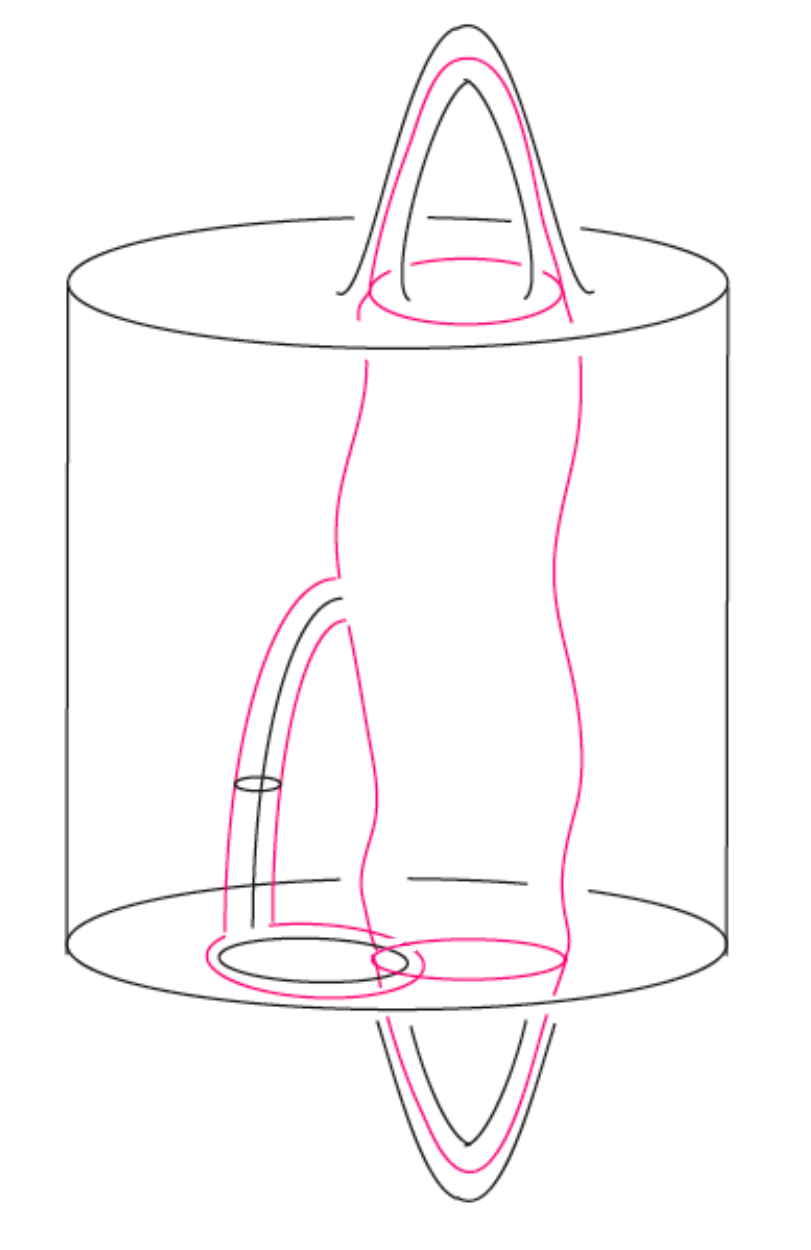}
\put(-100,130){$B_1$}
\caption{The new generator, $B_1$}
\label{fig:Tubing}
\end{figure}

Next, we consider the intersection of $B_k$ with the original generators $\{X_i,Y_i\}_{i=1}^r$ of $H_2(W, \bd W_-)$. Suppose that $B_k$ intersects some basis element $\Lambda \in \{X_i,Y_i\}$ nontrivially. Since $\Lambda$ can be chosen to be disjoint from $S_1'$, we employ the same methods as above. We take an arc $\alpha_1$ from $\Lambda$ to $\pm S_1'$ that doesn't intersect $B_k$ or the other generators and we consider the surface $\Lambda_1$, the result of performing ambient surgery along $\alpha_1$. Then, $\Lambda_1 \cdot B_k = \Lambda \cdot B_k \pm 1$. 

By repeating this process enough times, we come up with a basis $\{S_1', B_k, \Lambda_1, \dots, \Lambda_{2r}\}$ of $H_2(V, \bd V_-)$ with the intersection matrix 

\begin{center}
$
\left(
\begin{array}{ccccccc}
0 & 1 & 0 &0& \dots & 0&0 \\
1&0&0&0&\dots &0&0 \\
0&0&0&1& \dots &0&0 \\
0&0&1&0& \dots & 0&0\\
\vdots &\vdots &\vdots &\vdots &\dots&\vdots &\vdots \\
0&0&0&0& \dots &0&1 \\
0&0&0&0& \dots &1&0 \\
\end{array}
\right)
$
\end{center}
\end{proof}

These surfaces generate $H_2(V, \bd V_-)$ and they have the appropriate algebraic intersections. We can modify these surfaces by tubing along a Whitney arc to get a set of generators with the desired geometric intersection.

Finally, we note that $\hat{V}$, the closure of $V$, and $\hat{W}$, the closure of $W$ are the same manifold by construction. Since $\hat{W}$ is spin, then $\hat{V}$ is spin, and so $V$ is a 0-solve equivalence between $J$ and $J'$.

\section{Classification of Links up to 0-Solve Equiavlence}

We now give a classification of links up to 0-solve equivalence. As a corollary, we give necessary and sufficient conditions for a link to be 0-solvable. These results are consequences of the following theorem, which relates the condition of 0-solve equivalence to band-pass equivalence and to the Arf and Milnor's invariants. 

\begin{thm}\label{thm:Main}
For two ordered, oriented $m$-component links $\link$ and $\linkprime$ with vanishing pairwise linking numbers, the following conditions are equivalent:
\begin{enumerate}
\item $L$ and $L'$ are 0-solve equivalent,
\item $L$ and $L'$ are band-pass equivalent,
\item $\Arf(K_i) = \Arf(K_i')$ \\
$\bar{\mu}_L(ijk) = \bar{\mu}_{L'}(ijk)$\\
$\bar{\mu}_L(iijj) \equiv \bar{\mu}_{L'}(iijj) $ $(\bmod 2)$.
\end{enumerate}
\end{thm}

As an immediate corollary to this theorem, if $L'$ is the $m$-component unlink, we have the following result.

\begin{corollary}\label{cor:0solvable}
For an ordered, oriented $m$-component link $\link$ with vanishing pairwise linking numbers, the following conditions are equivalent:
\begin{enumerate}
\item $L$ is 0-solvable,
\item $L$ is 0-solve equivalent to the $m$-component unlink,
\item $L$ is band-pass equivalent to the $m$-component unlink,
\item $\Arf(K_i) = 0$ \\
$\bar{\mu}_L(ijk) = 0 $\\
$\bar{\mu}_L(iijj) \equiv 0 $ $(\bmod 2)$ .\\
\end{enumerate}
\end{corollary}

We will prove Theorem \ref{thm:Main} by a sequence of lemmas, first showing $(2) \Rightarrow (1)$, then showing that $(1) \Rightarrow (3)$, and finally, that $(3) \Rightarrow (2)$. 

\begin{lem}
If two ordered, oriented $m$-component links $\link$ and $\linkprime$ with vanishing pairwise linking numbers are band-pass equivalent, then $L$ and $L'$ are 0-solve equivalent. 
\end{lem}
 
 \begin{proof}
 
Because band-pass equivalence and 0-solve equivalence are both equivalence relations and thus transitive, we may assume that $L$ and $L'$ differ by a single band-pass move. Recall that we require the strands of each band to belong to the same link component. We first consider the 4-manifold $M_L \x \I $. Then, we attach two 0-framed 2-handles to $M_L \times \{1\}$ in the boundary of $M_L \times [0,1]$ along the attaching curves $\gamma_i$ and $\gamma_j$ pictured in figure \ref{fig:attach} below.

\begin{figure}[ht!]
\centering
\includegraphics[scale=.40]{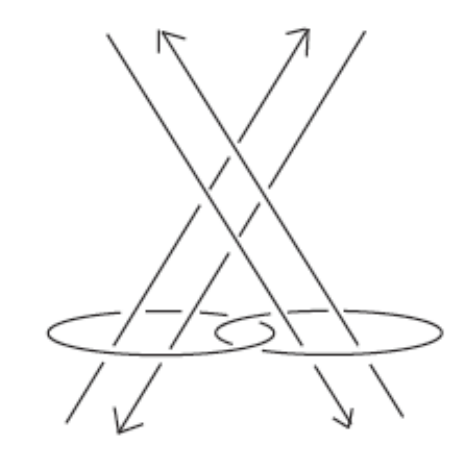}
\put(-150,50){$M_L \times \{1\}$}
\put(-14,80){\tiny0}
\put(-80,80){\tiny0}
\put(-27,60){$K_i$}
\put(-75,60){$K_j$}
\put(-90,30){$\gamma_i$}
\put(0,30){$\gamma_j$}
\put(-90,20){\tiny0}
\put(0,20){\tiny0}
\caption{Attaching curves in $M_L \x \I$}
\label{fig:attach}
\end{figure}

In $M_L \x \I$, we slide both strands of link component $K_i$ over the 2-handle attached to $\gamma_j$ and we slide both strands of link component $K_j$ over the 2-handle attached to $\gamma_i$. Note that, in figure \ref{fig:BPEslide}, the strands and attaching curves pictured each have a zero surgery coefficient; we perform the handle-slides in the closed manifold $M_L \x \{1\}$. 

\begin{figure}[ht!]
\centering
\includegraphics[scale=.35]{BPEslide1.pdf}
\includegraphics[scale=.35]{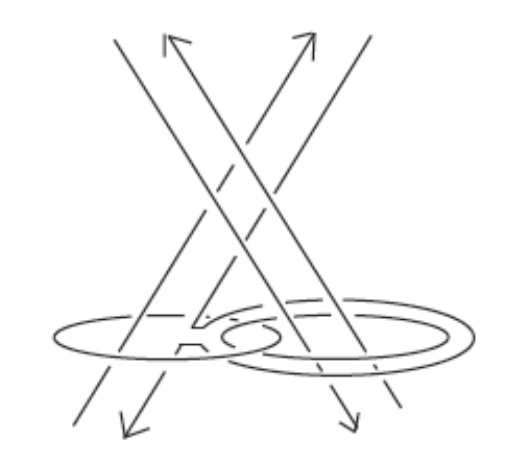}
\includegraphics[scale=.35]{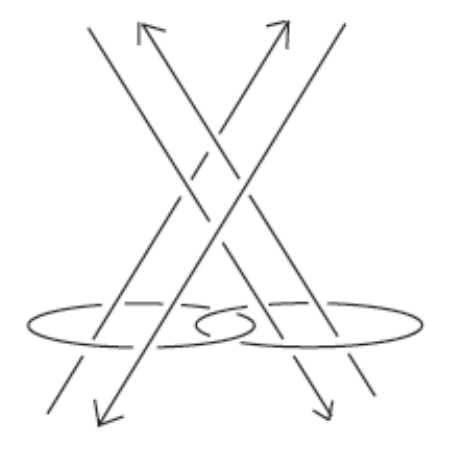}
\includegraphics[scale=.35]{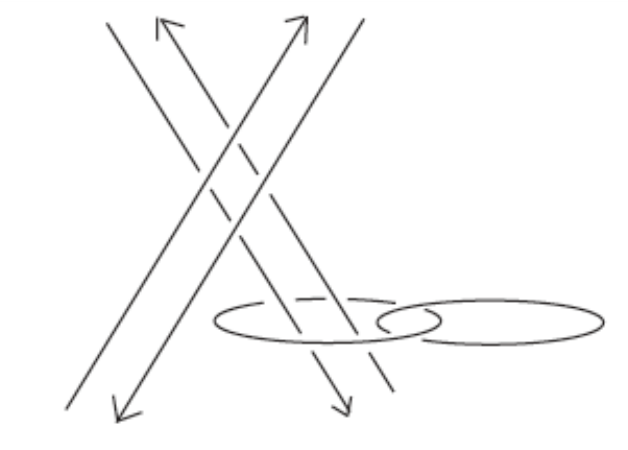}
\includegraphics[scale=.35]{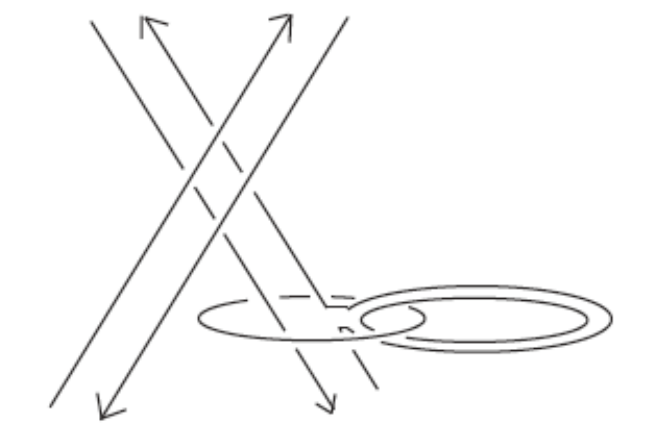}
\includegraphics[scale=.35]{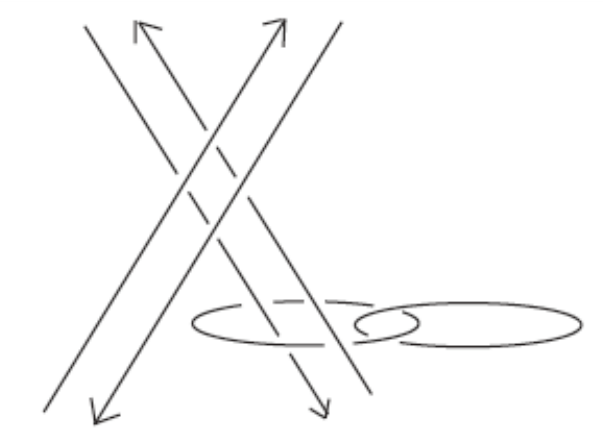}
\includegraphics[scale=.35]{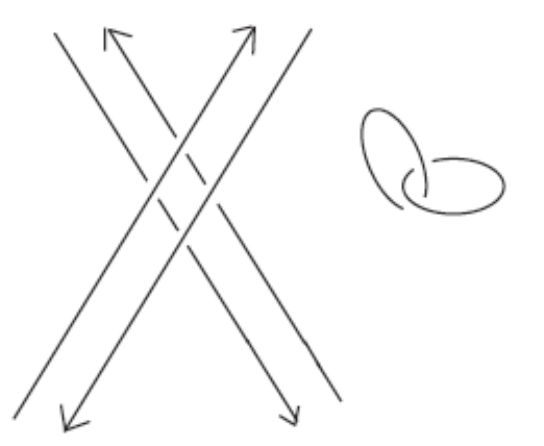}
\caption{Performing a Band-Pass via Handleslides}
\label{fig:BPEslide}
\end{figure}
 
The resulting 3-manifold, after sliding all four strands, is $M_{L'}$, as 0-framed surgery on the Hopf link is homeomorphic to $S^3$. Therefore, we consider the 4-manifold $W = M_L \times [0,1] \cup \{$two 2-handles$\}$, and we see that $W$ is a cobordism between $M_L$ and $-M_{L'}$. We wish to show that $W$ is a 0-solve equivalence.

\begin{figure}[ht!]
\centering
\includegraphics[scale=.30]{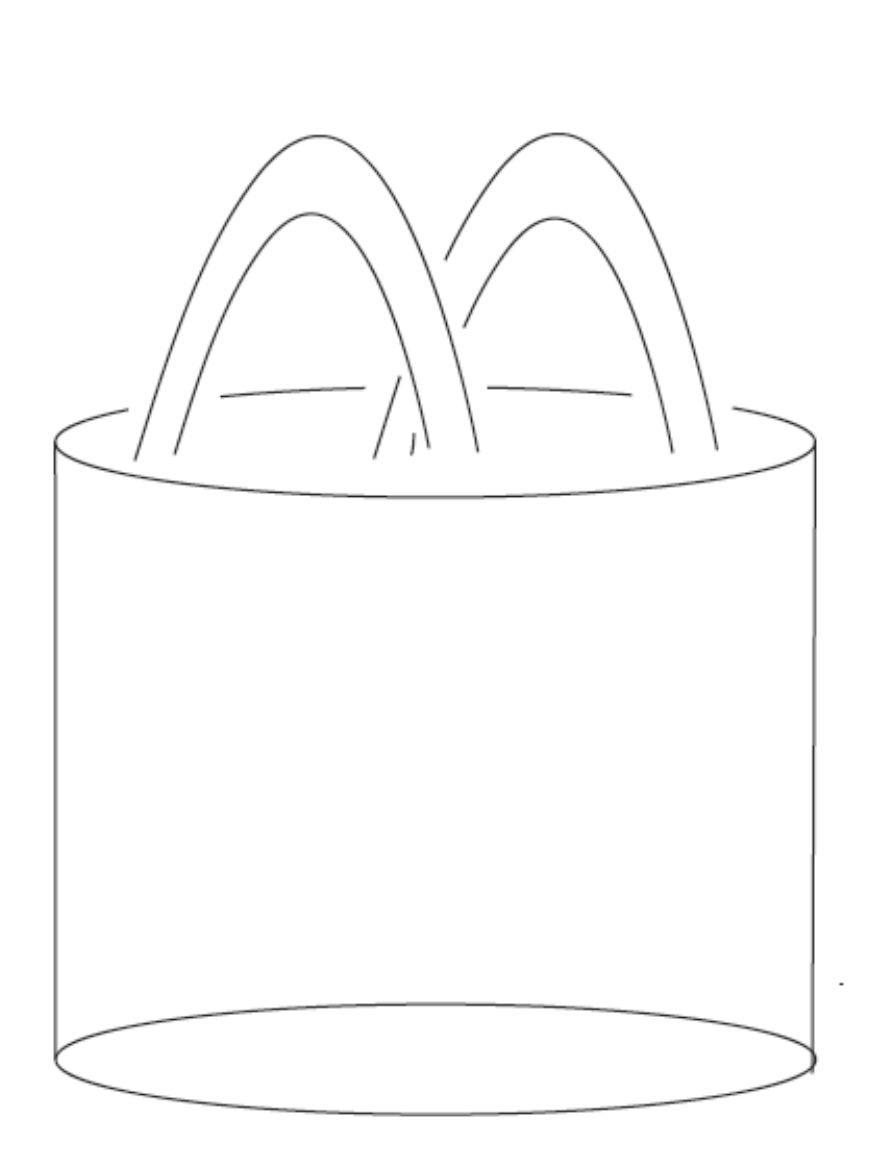}
\put(-180,50){W=}
\put(0,10){$M_L$}
\put(-10,130){$-M_{L'}$}
\caption{A cobordism between $M_L$ and $-M_{L'}$}
\end{figure}

The attaching curves $\gamma_i$ and $\gamma_j$ are null-homologous due to the requirement that the strands of each band belong to the same component. Thus, attaching the 2-handles has no effect on $H_1$ so we see that $i_*:H_1(M_{L'}) \rightarrow H_1(W)$ and $j_*: H_1(M_L) \rightarrow H_1(W)$ are isomorphisms. 

Comparing the second homology of $W$ to that of $M_L \x \I$, we see that $H_2(W) \cong H_2(M_L \x [0,1]) \oplus \Z \oplus \Z \cong H_2(M_L) \oplus \Z \oplus \Z$ with the added homology generated by $\hat{\Lambda}_1$ and $\hat{\Lambda}_2$, where $\hat{\Lambda}_i$ is the surface $\Lambda_i$ pictured in figure \ref{fig:surfaces} capped off with a 2-cell from the attached 2-handles. 

\begin{figure}[ht!]
\centering
\includegraphics[scale=.45]{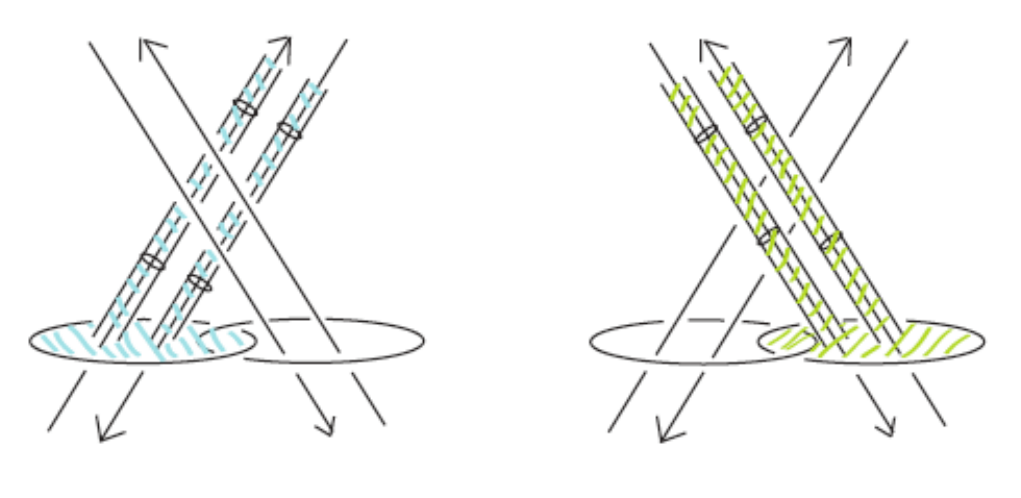}
\put(-210,50){$\Lambda_1$}
\put(-30,50){$\Lambda_2$}
\caption{New Generators for $H_2(W)$}
\label{fig:surfaces}
\end{figure}

By pushing the interior of $\Lambda_1$ very slightly into $\operatorname{int}(M_L \x \I) \subset W$, we can assure that  $\hat{\Lambda}_1 \cap \hat{\Lambda}_2 = \{p\}$, where $p$ is a point. We choose orientations on $\hat{\Lambda}_1$ and $\hat{\Lambda}_2$ to ensure that $\hat{\Lambda}_1 \cdot \hat{\Lambda}_2 = +1$. We can also assure that $\hat{\Lambda}_1$ and $\hat{\Lambda}_2$ are disjoint from the generators of $H_2(M_L \x \I)$. Therefore, $H_2(W, \bd W_-) \cong \frac{H_2(W)}{H_2(\partial{W}_-)}$ has basis $\{\hat{\Lambda}_1, \hat{\Lambda}_2\}$. Each of $\Lambda_1$ and $\Lambda_2$ are surfaces with boundary and thus have trivial normal bundles inducing the 0 framing on their boundaries. The handle is attached with the 0 framing, so $\hat{\Lambda}_1$ and $\hat{\Lambda}_2$ have trivial normal bundles. The second condition of being a 0-solve equivalence is satisfied.

The generators of $H_2(\hat{W})$ are \linebreak $\{ \hat{\Sigma}_1, \dots, \hat{\Sigma}_m, F_1, \dots, F_m, \hat{\Lambda}_1, \hat{\Lambda}_2\}$, where $\hat{\Sigma}_i$ is an oriented Seifert surface for $K_i$, closed off with a disk in $M_L \x \{0\}$, $F_i$ is the annulus $\mu_i \x \I$ (where $\mu_i$ is the meridian of $K_i$) closed off with disks from the process of closing $W$, and $\hat{\Lambda}_i$ are the surfaces mentioned above. We can assure that the $\hat{\Lambda}_i$ are disjoint from the $F_i$, and so the intersection matrix for $Q_{\hat{W}} = \bigoplus \left(
\begin{array}{cc}
0&1\\
1&0\\
\end{array} \right) $ and so $\hat{W}$ is spin.

These conditions show that $W$ is a 0-solve equivalence between $L$ and $L'$. 
 
 \end{proof}

Next, we prove the second step of Theorem \ref{thm:Main} using the following results. 

\begin{lem}\label{Arf}
 Suppose that $\link$ and $\linkprime$ are two ordered, oriented, $m$-component, 0-solve equivalent links with vanishing pairwise linking numbers. Then, $\Arf(K_i) = \Arf(K_i')$.
 \end{lem}
 
 \begin{proof}
 By Proposition \ref{prop:sublinks}, for all $i$, $K_i$ and $K_i'$ are 0-solve equivalent knots. By Proposition \ref{prop:0solv}, either both $K_i$ and $K_i'$ are 0-solvable knots, or neither $K_i$ nor $K_i'$ are 0-solvable knots. If both $K_i$ and $K_i$ are 0-solvable knots, we know that $\Arf(K_i) = \Arf(K_i') = 0.$ If neither $K_i$ nor $K_i'$ are 0-solvable knots, we know that $\Arf(K_i) = \Arf(K_i') = 1$ \cite{COT}. Either way, we must have that $\Arf(K_i) = \Arf(K_i')$.
 \end{proof}
 

 Let $J$ and $J'$ be ordered, oriented, 0-solve equivalent links with vanishing pairwise linking numbers, and suppose that $J$ and $J'$ are $m$-component links where $m=2$ or $m=3$. Let $W$ be a 4-manifold realizing the 0-solve equivalence from $J$ to $J'$. We can choose oriented Seifert surfaces $\{\Sigma_k\}_{k=1}^m$ in $S^3 - J$ for each component of $J$ in such a way that each pair of Seifert surfaces intersect in a single simple closed curve \cite{Tim}. We choose oriented Seifert surfaces $\{\Sigma_k' \}_{k=1}^m$ for the components of $J'$ in $S^3 - {J'}$ subject to the same conditions. 
 
 In $M_J$ and $M_{J'}$, each Seifert surface $\Sigma_k$ and $\Sigma_k'$ is capped off with a disk yielding closed surfaces $\{\hat{\Sigma}_k\}$ in $M_J$ and $\{\hat{\Sigma}_k'\}$ in $M_{J'}$. Within each 3-manifold, the surfaces intersect pairwise in a circle.

Using the Thom-Pontryagin construction, we define maps $f_k: M_J \rightarrow S^1 \subset \mathbb{C}$ as follows. Take a product neighborhood $\hat{\Sigma}_k \x [-1,1]$, where the $+1$ corresponds to the positive side of $\hat{\Sigma}_k$. Define $f_k: \hat{\Sigma}_k \x [-1,1] \rightarrow S^1$ by $(x,t) \mapsto e^{2\pi i t}$, and for $y \in M_J - (\hat{\Sigma}_k \x [-1,1])$, define $f_k(y) = -1$. Similarly define maps $f_k' : -M_{J'} \rightarrow S^1 \subset \mathbb{C}$. Then, we consider the maps $f = f_1 \x \dots \x f_m: M_J \rightarrow (S^1)^m$ and $f' = f'_1 \x \dots \x f'_m: -M_{J'} \rightarrow (S^1)^m$. The maps $\{\pi_k\}$ are the standard projection maps, so that $f_k = \pi_k \circ f$ and $f_k' = \pi_k \circ f'$.

\begin{figure}[h]
\centering
\includegraphics[scale=.35]{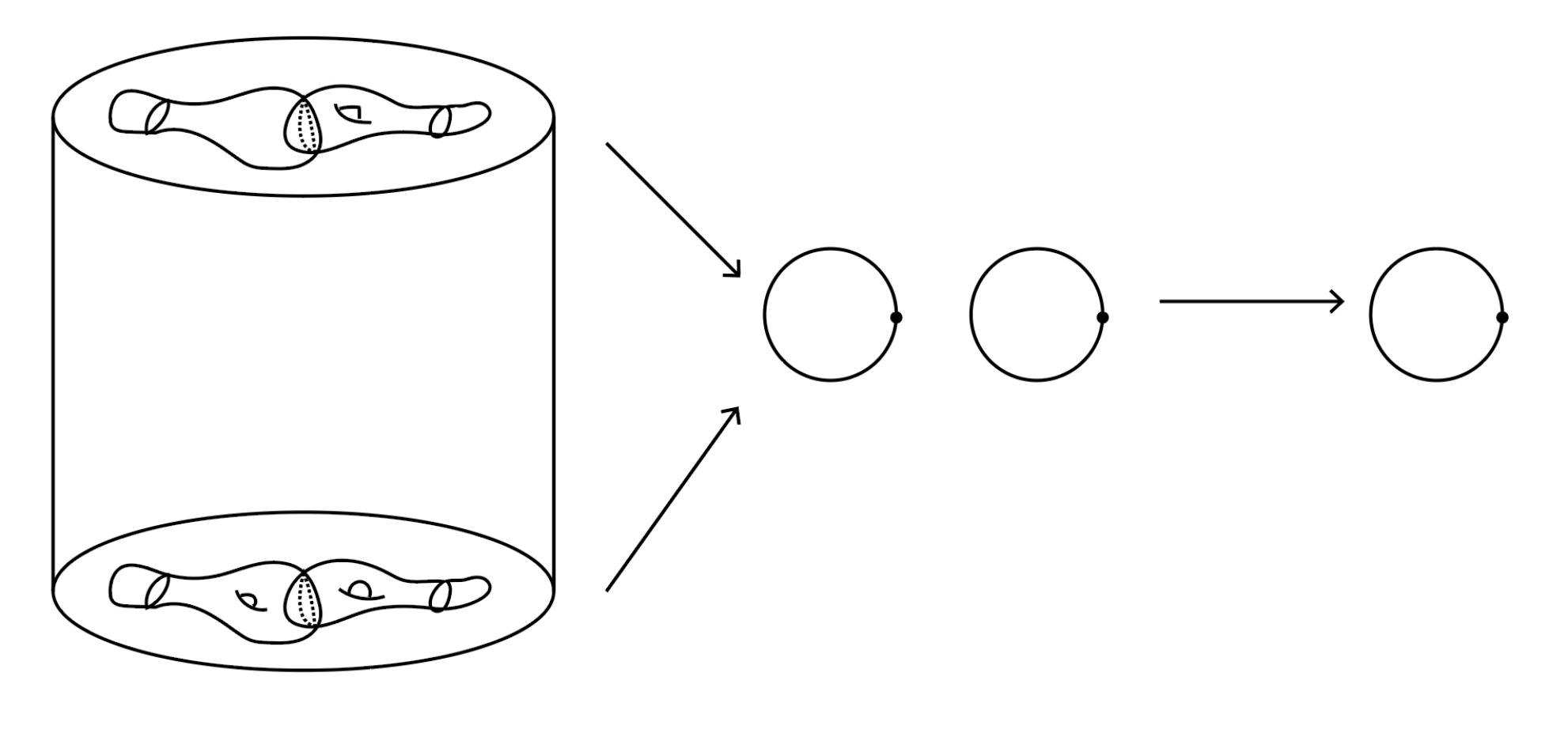}
\put(-138,87){$\times$}
\put(-190,119){$f$}
\put(-190,40){$f'$}
\put(-75,100){$\pi_k$}
\put(-142,95){$1$}
\put(-98,95){$1$}
\put(-14,95){$1$}
\put(-300,152){$\hat{\Sigma_1}$}
\put(-250,152){$\hat{\Sigma_2}$}
\put(-300,4){$\hat{\Sigma_1'}$}
\put(-250,4){$\hat{\Sigma_2'}$}
\put(-265,118){$A$}
\put(-275,38){$A'$}
\caption{The Thom-Pontryagin Construction on $M_J$ and $-M_{J'}$}
\end{figure}

%

We also have a framing on the surfaces $\hat{\Sigma}_k$ and $\hat{\Sigma}_n'$ given by the normal direction to the tangent plane $T_x\hat{\Sigma}_k$ at each point $x \in \hat{\Sigma}_k$. The maps $f:M_J \rightarrow (S^1)^m$ and $f': -M_{J'} \rightarrow (S^1)^m$ induce maps on first homology, $f_*: H_1(M_J) \rightarrow H_1((S^1)^m)$ and $f'_*:H_1(-M_{J'}) \rightarrow H_1((S^1)^m)$ that are isomorphisms. As $W$ is a 0-solve equivalence, the induced inclusion maps $i_*$ and $j_*$ are also isomorphisms. 



\[
\begin{diagram}
\node{\pi_1(M_J)} \arrow{e}
\node{H_1(M_J)} \arrow{s,lr}{\cong}{i_*} \arrow{ese,tb}{f_*}{\cong}
\\
\node{\pi_1(W)} \arrow{e}
\node{H_1(W)} \arrow[2]{e,t,3,..}{\alpha}
   \node[2]{H_1((S^1)^m)}
\\
\node{\pi_1(-M_{J'})} \arrow{e}
\node{H_1(-M_{J'})} \arrow{n,lr}{\cong}{j_*} \arrow{ene,tb}{\cong}{f'_*}  
\end{diagram}
\]
\\

Then, define the map $\alpha: H_1(W) \rightarrow H_1((S^1)^m)$ to be $\alpha = f_* \circ i_*^{-1} = f'_* \circ j_*^{-1}$, by the Thom-Pontryagin construction and the fact that $i_*(\mu_k) = j_*(\mu'_k)$. Since $\pi_1((S^1)^m) \cong \Z^m$ is abelian, we can define the map $\bar{\alpha}: \pi_1(W) \rightarrow \pi_1((S^1)^m)$ as in the following diagram:

\[
\begin{diagram}
\node{\pi_1(W)}\arrow{e,t}{\bar{\alpha}}\arrow{s}
\node{\pi_1((S^1)^m)}\arrow{s,r}{\cong}
\\
\node{H_1(W)}\arrow{e,t}{\alpha}
\node{H_1((S^1)^m)}
\end{diagram}
\]
\\

We wish to show that there is an extension $\bar{f}: W \rightarrow (S^1)^m$ such that $\bar{f}|_{\bd M_J} = f$, $\bar{f}|_{\bd -M_{J'}}=f'$, and $\bar{f}_*=\bar{\alpha}$. 

We have the following commutative diagram:

\[
\begin{diagram}
\node{}
\node{\pi_1(M_J)}\arrow{sw,t}{i_*}\arrow{se,t}{f_*}
\\
\node{\pi_1(W)}\arrow[2]{e,t}{\bar{\alpha}}
\node{}
\node{\pi_1((S^1)^m)}
\\
\node{}
\node{\pi_1(-M_{J'})}\arrow{nw,b}{j_*}\arrow{ne,b}{f'_*}
\\
\end{diagram}
\]

%

As the CW-complex $S^1 \x \dots \x S^1$ is an Eilenberg-Maclane space $K(\Z^m,1)$, we can extend the maps $f, f'$ to a map $\bar{f}:W \rightarrow S^1 \x \dots \x S^1$ such that $\bar{f}|_{M_J} = f$, $\bar{f}|_{-M_{J'}}=f'$, and $\bar{f}_*=\bar{\alpha}$. Up to a homotopy, we can arrange that $\bar{f}$ is transverse to the point $\vec{1}$ in $(S^1)^m$. Then, the preimage $\bar{f}^{-1}\{\vec{1}\}$ in $W$ is a framed submanifold cobound by the intersections $\bigcap \hat{\Sigma}_k$ and $\bigcap \hat{\Sigma}_k'$. We use the cobordisms in the following lemmas.


 \begin{lem}\label{ijk}
 Suppose that $\link$ and $\linkprime$ are two ordered, oriented, $m$-component, 0-solve equivalent links with vanishing pairwise linking numbers. Then, $\bar{\mu}_L(ijk) = \bar{\mu}_{L'}(ijk)$.
 \end{lem}
 
 \begin{proof}

Let $J$ and $J'$ be corresponding 3-component sublinks $J = K_i \cup K_j \cup K_k$ and $J' = K_i' \cup K_j' \cup K_k'$ of $L$ and $L'$, respectively. For simplicity, we use the indices $1,2,3$  for $J$ and $J'$ to denote the ordering of the components. Thus, $\bar{\mu}_L(ijk) = \bar{\mu}_J(123)$ and $\bar{\mu}_{L'}(ijk) = \bar{\mu}_{J'}(123)$. By Proposition \ref{prop:sublinks}, $J$ and $J'$ are 0-solve equivalent links; we wish to show that $\bar{\mu}_J(123) = \bar{\mu}_{J'}(123)$. 

Using the above construction in the case where $m=3$, the 1-manifold $\bar{f}^{-1}\{(1,1,1)\}$ in $W$ is a framed cobordism between $f^{-1}\{(1,1,1)\} = \hat{\Sigma_1} \cap \hat{\Sigma_2} \cap \hat{\Sigma_3} \subset M_J$ and $f'^{-1}\{(1,1,1)\}= \hat{\Sigma_1'} \cap \hat{\Sigma_2'} \cap \hat{\Sigma_3'} \subset M_{J'}$, which are each a collection of triple points. Respecting the ordering of the components, for each point $p_n \in f^{-1}\{(1,1,1)\}$ and for each point $p_n' \in f'^{-1}\{(1,1,1)\}$, we assign a sign of $+1$ if the orientation of $p_n$ (respectively, $p_n'$) induced by the framings on the surfaces agrees with the orientation on $M_J$ (respectively, $-M_{J'}$), and a sign of $-1$ otherwise. Then, the Milnor's invariant $\bar{\mu}_J(123) = \sum {(-1)^{\epsilon_n}}$, where $\epsilon_n = \pm 1$ is the sign of the point $p_n$. The Milnor's invariant $\bar{\mu}_{J'}(123) = \sum {(-1)^{\epsilon_n'}}$, where $\epsilon_n' = \pm 1$ is the sign of the point $p_n'$ \cite{Tim}. As the collections $\{p_n\}$ and $\{p_n'\}$ cobound a framed 1-manifold, $\bar{\mu}_J(123)$ and $\bar{\mu}_{J'}(123)$ must be equal.

\begin{figure}[ht!]
\centering
\includegraphics[scale=.3]{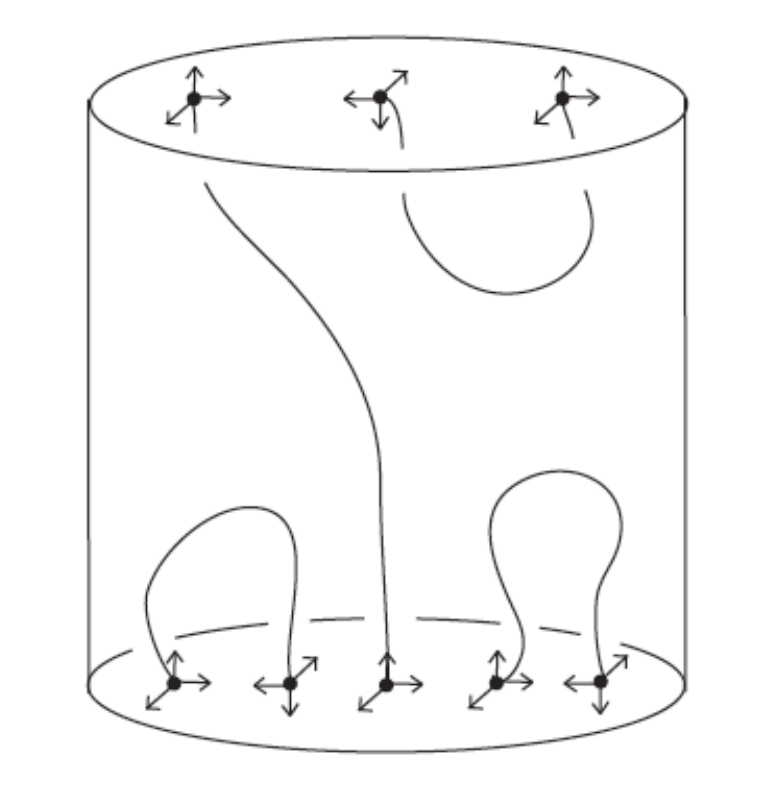}
\put(-170,60){$\bar{f}^{-1}\{(1,1,1)\}$}
\caption{Framed Cobordism between $f^{-1}\{(1,1,1)\}$ and $f'^{-1}\{(1,1,1)\}$}
\end{figure}

 \end{proof}

 \begin{lem}\label{iijj}
 Suppose that $\link$ and $\linkprime$ are two ordered, oriented, $m$-component, 0-solve equivalent links with vanishing pairwise linking numbers. Then, $\bar{\mu}_L(iijj) \equiv \bar{\mu}_{L'}(iijj) $ $(\bmod 2)$.
 \end{lem}

\begin{proof}
Let $J$ and $J'$ be corresponding 2-component sublinks $J = K_i \cup K_j$ and $J' = K_i' \cup K_j'$ of $L$ and $L'$, respectively. For simplicity, we use the indices $1,2$ for $J$ and $J'$ to denote the ordering of the components. Thus, $\bar{\mu}_L(iijj) = \bar{\mu}_J(1122)$ and $\bar{\mu}_{L'}(iijj) = \bar{\mu}_{J'}(1122)$ By Proposition \ref{prop:sublinks}, $J$ and $J'$ are 0-solve equivalent links; we wish to show that $\bar{\mu}_J(1122) \equiv \bar{\mu}_{J'}(1122) \bmod 2$.

From the Thom-Pontryagin construction with $m=2$, the pre-image $\bar{f}^{-1}\{(1,1)\}$ is a framed surface in $W$ with boundary $A \bigsqcup -A'$ in $\bd W$. In $S^1 \x S^1$, we consider $D_{p}$, an $\epsilon$- neighborhood of the point $p = (1,1)$. Let $q=(e^{2\pi i \epsilon}, 1)$ be a point on the boundary of $D_p$. Then, $\bar{f}^{-1}(q)$ is a surface in $W$ with boundary $B \bigsqcup -B'$, where $B$ is the curve $A$ pushed off in the positive direction of $\hat{\Sigma}_1$ and $B'$ is the curve $A'$ pushed off in the positive direction of $\hat{\Sigma}_1'$. See figure \ref{fig:1122map}. Furthermore, these surfaces $\bar{f}^{-1}(p)$ and $\bar{f}^{-1}(q)$ are disjoint in $W$, and if we consider a path $\gamma (t) = (e^{2 \pi i \epsilon t };  1), t \in \I $, the pre-image $\bar{f}^{-1}( \gamma[0,1] )$ is a 3-submanifold of $W$ that is cobounded by the surfaces $\bar{f}^{-1}(p)$ and $\bar{f}^{-1}(q)$. For a sufficiently small choice of $\epsilon >0$, we can ensure that $\bar{f}^{-1}(\gamma[0,1])$ is disjoint from the meridians of $J$ and $J'$. 

\begin{figure}[ht!]
\centering
\includegraphics[scale=.4]{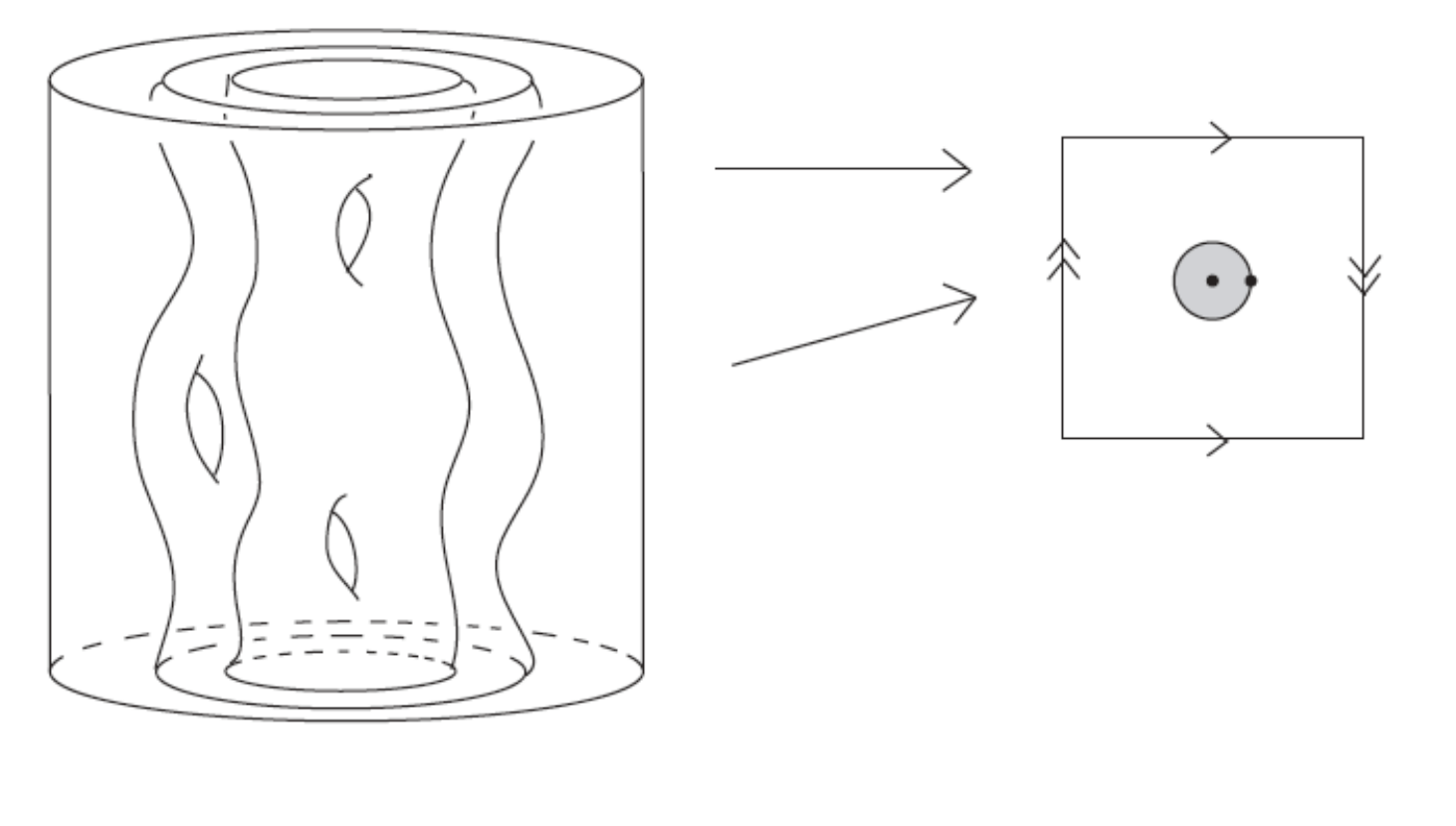}
\put(-255,155){$B$}
\put(-240,155){$A$}
\put(-255,10){$B'$}
\put(-240,10){$A'$}
\put(-120,133){$f$}
\put(-120,80){$\bar{f}$}
\put(-65,90){$D_p$}
\put(-40,110){\tiny{$q$}}
\caption{Surface pre-images under the map $\bar{f}$}
\label{fig:1122map}
\end{figure}

The Sato-Levine invariants of $J$ and $J'$, are the quantities $\bar{\mu}_J(1122) = \lk(A,B)$ and $\bar{\mu}_{J'}(1122) = \lk(A',B')$. Therefore, we wish to show that $\lk(A,B) \equiv \lk(A',B') $ $(\bmod 2)$. 

We construct the closure $\hat{W}$ from $W$ by first attaching 0-framed 2-handles to the meridians $\mu_1$ and $\mu_2$ of link $J$ and to meridians $\mu_1'$ and $\mu_2'$ of link $J'$ and then attaching a 4-handle $\B^4$ to the $S^3$ boundary component and a 0-handle $\B^4$ to the $-S^3$ boundary component. 

The curves $A$ and $B$ each bound surfaces in the 4-handle $B^4$. We call these surfaces $S$ and $T$. Similarly, the curves $A'$ and $B'$ each bound surfaces in the 0-handle $B^4$. We call these surfaces $S'$ and $T'$. Noting that $\bar{\mu}_J(1122) = \lk(A,B) = S \cdot T$ and $\bar{\mu}_{J'}(1122) = \lk(A',B')=S' \cdot T'$, we wish to show that $S \cdot T + S' \cdot T' \equiv 0 $ $(\bmod 2)$. Let $\hat{S} = S \cup_A \bar{f}^{-1}(p) \cup_{A'} S'$ and $\hat{T} = T \cup_B \bar{f}^{-1}(q) \cup_{B'} T'$ be closed surfaces in $\hat{W}$ as pictured in figure \ref{fig:21}. Because $\bar{f}^{-1}(p) \cap \bar{f}^{-1}(q) = \emptyset$, the problem reduces to showing that $\hat{S} \cdot \hat{T} \equiv 0 $ $(\bmod 2)$. 

\begin{figure}[ht!]
\centering
\includegraphics[scale=.6]{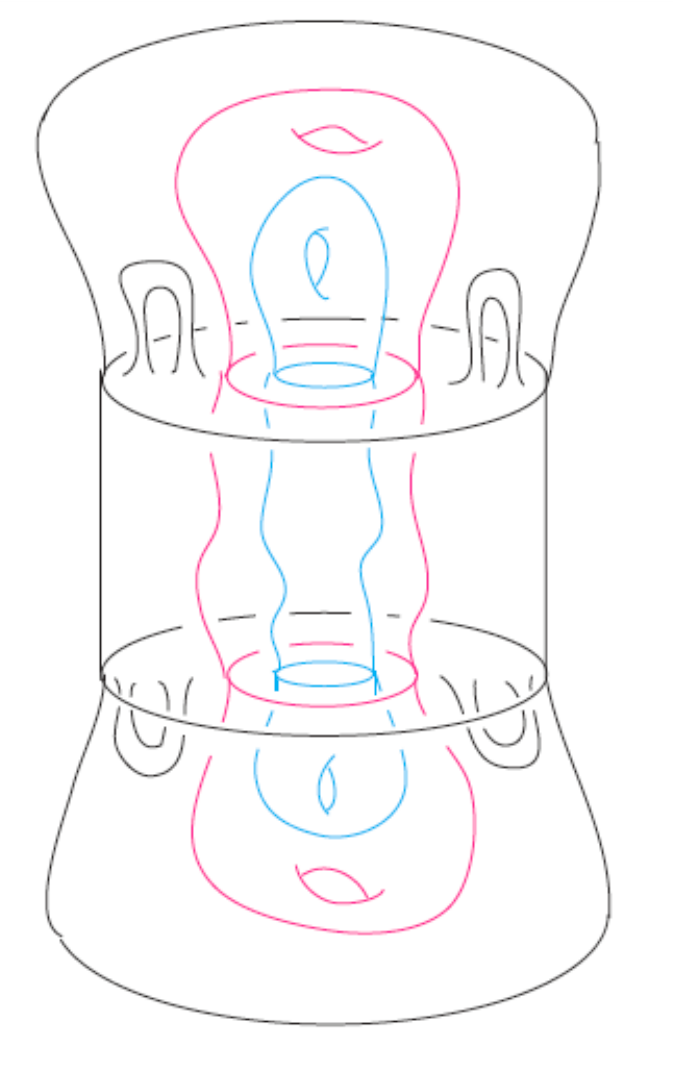}
\put(-130,240){$\hat{S}$}
\put(-155,260){$\hat{T}$}
\caption{Surfaces $\hat{S}$ and $\hat{T}$ in $\hat{W}$}
\label{fig:21}
\end{figure}

The surface $S \cup_A f^{-1}(\gamma (t)) \cup_B T \subset \B^4$ is a closed surface, and therefore is the boundary of a 3-chain. Similarly, the surface $S' \cup_{A'} f'^{-1}(\gamma (t)) \cup_{B'} T' \subset \B^4$ is a closed surface, and therefore is the boundary of a 3-chain. Piecing together these 3-chains, along with a 3-chain representing $\bar{f}^{-1}(\gamma[0,1])$, we see that $\hat{S}$ and $\hat{T}$ cobound a 3-chain, and so the homology classes $[\hat{S}] \in H_2(\hat{W})$ = $[\hat{T}] \in H_2(\hat{W})$ are equal. Finally, we recall that $\hat{W}$ is spin and $H_1(\hat{W})$ has no 2-torsion, so the intersection form $Q_{\hat{W}}$ in even. Thus, we have that $\hat{S} \cdot \hat{T} = Q_{\hat{W}}([\hat{S}],[\hat{T}]) = Q_{\hat{W}}([\hat{S}],[\hat{S}])$ is even. This tells us that $\bar{\mu}_J(1122) \cong \bar{\mu}_{J'}(1122) $ $(\bmod 2)$.
\end{proof}

Finally, we conclude the proof of Theorem \ref{thm:Main} by showing the following lemma.
 
 \begin{lem}\label{TandY}
 If $\link$ and $\linkprime$ are two ordered, oriented, $m$-component links with vanishing pairwise linking numbers such that
  \begin{enumerate}
 \item $\Arf(K_i)= \Arf(K_i')$, 
 \item $\bar{\mu}_L(ijk) = \bar{\mu}_{L'}(ijk)$,
 \item $\bar{\mu}_L(iijj) \equiv \bar{\mu}_{L'}(iijj) (\bmod{2}),$
 \end{enumerate}
 then $L$ and $L'$ are band-pass equivalent.
 \end{lem}
 
 \begin{proof}
 We will use the result and proof of the following theorem due to Taniyama and Yasuhara \cite{TY}.
 
 \begin{thm}{\emph{[Taniyama-Yasuhara] \cite{TY} }}
 Let $\link$ and $\linkprime$ be ordered, oriented, $m$-component links with vanishing pairwise linking numbers. The following conditions are equivalent:
\begin{enumerate}
\item $L$ and $L'$ are clasp-pass equivalent links,
\item $a_2(K_i) = a_2(K'_i)$\\
$a_3(K_i \cup K_j) \equiv a_3(K_i' \cup K_j') $ $(\bmod 2)$ \\ 
and $\bar{\mu}_L(ijk) = \bar{\mu}_{L'}(ijk),$
\end{enumerate}
\end{thm}

Here, $a_j$ indicates the $j^{th}$ coefficient of the Conway polynomial, as defined in \cite{Wu}. Taniyama and Yasuhara prove the direction (2) $\Rightarrow$ (1). In the following proposition, we show that the three assumptions in Lemma \ref{TandY}, $\Arf(K_i)= \Arf(K_i')$, 
 $\bar{\mu}_L(ijk) = \bar{\mu}_{L'}(ijk)$, and
$\bar{\mu}_L(iijj) \equiv \bar{\mu}_{L'}(iijj) (\bmod{2}),$  are very similar to Taniyama and Yasuhara's assumptions that $a_2(K_i) = a_2(K'_i)$,
$a_3(K_i \cup K_j) \equiv a_3(K_i' \cup K_j') $ $(\bmod 2)$,
and $\bar{\mu}_L(ijk) = \bar{\mu}_{L'}(ijk).$ 

\begin{proposition}\label{prop:Conway}
Let $\link$ be an ordered, oriented, $m$-component link with vanishing pairwise linking numbers. Then,
\begin{enumerate}
\item $\Arf(K_i)$ is given by $a_2(K_i) ( \bmod{2} )$ (Kauffman, \cite{Kauff}), and
\item $\bar{\mu}_L(iijj) = a_3(K_i \cup K_j) $.
\end{enumerate}
\end{proposition}

\begin{proof}
Here, $a_j$ refers to the coefficient of the term $z^j$ in the Conway polynomial of a link, $\Delta_L(z) = a_0 + a_1z + a_2z^2 + a_3z^3 + \dots$, as defined in \cite{Wu}. In \cite{Kauff}, Kauffman shows that for a knot $K$, $a_2(K) (\bmod 2) = \Arf(K)$. 

In \cite{Tim2}, Cochran writes the general Conway polynomial of a link $L$ as $\Delta_L(z) = z^{m-1}(b_0 + b_2z^2 + \dots + b_{2n}z^{2n})$ where $m$ is the number of components of $L$. Using Cochran's notation for 2-component sublink $J = K_i \cup K_j$ of $L$, $\Delta_J(z) = z(b_0 + b_2z^2 + \dots +b_{2n}z^{2n})$. Therefore, the cubic term of the Conway polynomial has coefficient $b_2$, so $a_3(J)=b_2(J)$. Then Corollary 4.2 in \cite{Tim2} gives that, for a 2-component link $J$, the coefficient $b_0(J) = -\bar{\mu}_J(12)$, and if $b_0(J) = 0$, then $b_2(J)= \bar{\mu}_J(1122)$. Since we are assuming that all pairwise linking numbers of $L$ vanish, we have that $\bar{\mu}_J(1122) = b_2(J) = a_3(J)$. Thus, for all $i$ and $j$, $\bar{\mu}_L(iijj) =a_3(K_i \cup K_j)$.
\end{proof}

To complete the proof of Lemma \ref{TandY}, we adapt Taniyama and Yasuhara's proof to show the following proposition.

\begin{proposition}\label{prop:TY}
For two ordered, oriented, $m$-component links $\link$ and $\linkprime$ with vanishing pairwise linking numbers, if 
\begin{enumerate}
\item $a_2(K_i) \equiv a_2(K_i') (\bmod 2)$,
\item$a_3(K_i \cup K_j) \equiv a_3(K_i' \cup K_j') (\bmod{2})$, and
\item $\bar{\mu}_L(ijk) = \bar{\mu}_{L'}(ijk)$,
\end{enumerate}

then $L$ and $L'$ are band-pass equivalent.
\end{proposition}

\begin{proof}
Given that all pairwise linking numbers of $L$ and $L'$ vanish, Taniyama and Yasuhara show that that $L$ and $L'$ can both be obtained from $U^m$ by inserting a sequence of Borromean rings, as pictured in the following skein relation. 

\begin{figure}[ht!]
\centering
\includegraphics[scale=.4]{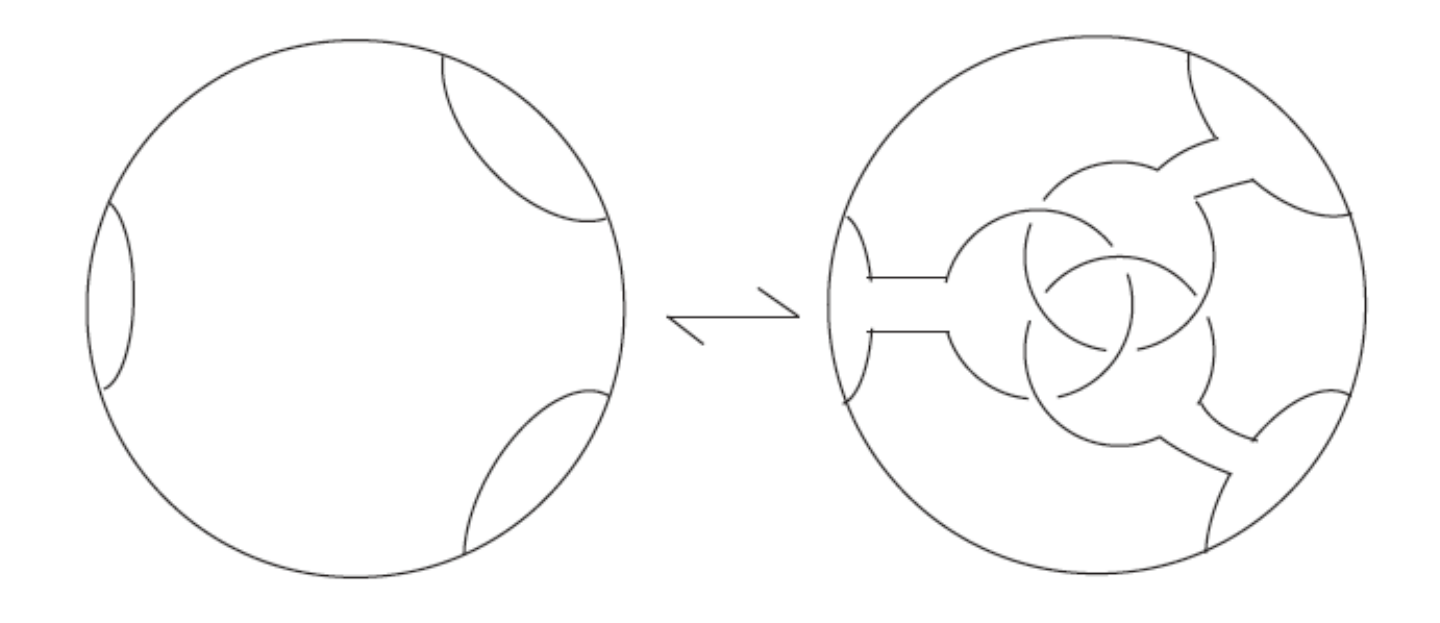}
\caption{A Borromean Rings Insertion}
\end{figure}

We note that the feet of the Borromean rings insertion can attach to any part of the link. In this way, a Borromean rings insertion may involve one, two, or three distinct link components. 

\begin{definition}{[Taniyama-Yasuhara]}

A \emph{Borromean chord} $C$ is a neighborhood of a 3-ball containing a Borromean insertion and a neighborhood of its attaching bands. 

\begin{figure}[ht!]
\centering
\includegraphics[scale=.4]{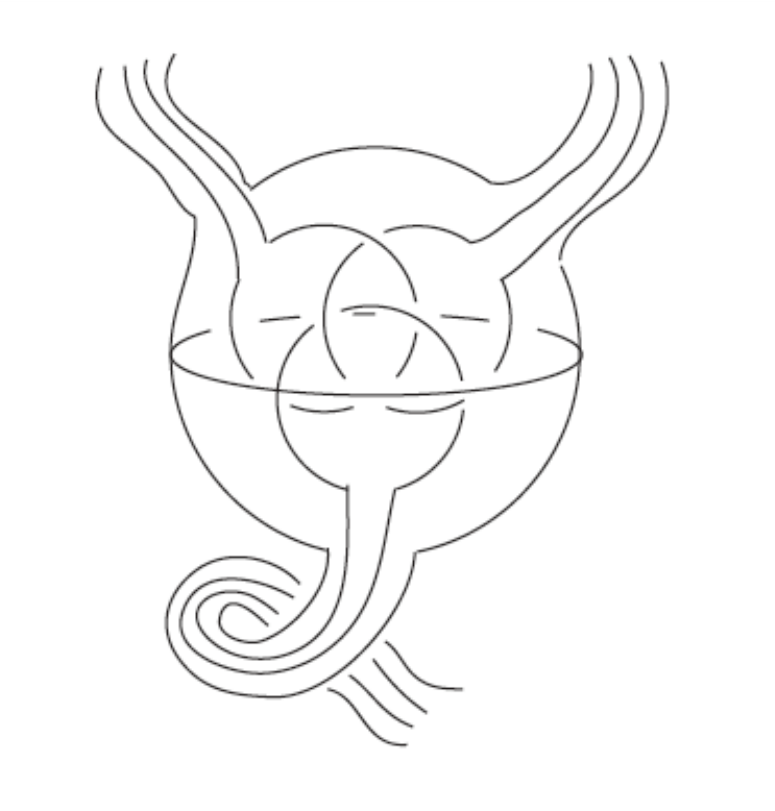}
\caption{Borromean Chord}
\end{figure}

For distinct $i,j,k$, we say that a chord $C$ is of \emph{type (i)} if each of the bands in the Borromean ring insertion attach to the $i^{th}$ link component. $C$ is of \emph{type (ij)} if the bands attach to the $i^{th}$ and $j^{th}$ components, and $C$ is of \emph{type (ijk)} if the bands attach to the $i^{th}, j^{th},$ and $k^{th}$ link components.
\end{definition}

In \cite{TY} Lemma 2.5, Taniyama and Yasuhara show that any ordered, oriented, $m$-component link $L$ is clasp-pass equivalent to an ordered, oriented, $m$-component link $J$, where $J$ is formed from $U^m$ by Borromean ring insertion, where $J$ is of the following form.

\begin{enumerate}

\item Each Borromean chord of $J$ of type $(ijk)$ is contained in a 3-ball as illustrated in Figure \ref{fig:typeijk} (a) or (b), and for each set of components $i,j,k$, not both (a) and (b) occur.

\begin{figure}[ht!]
\centering
\includegraphics[scale=.35]{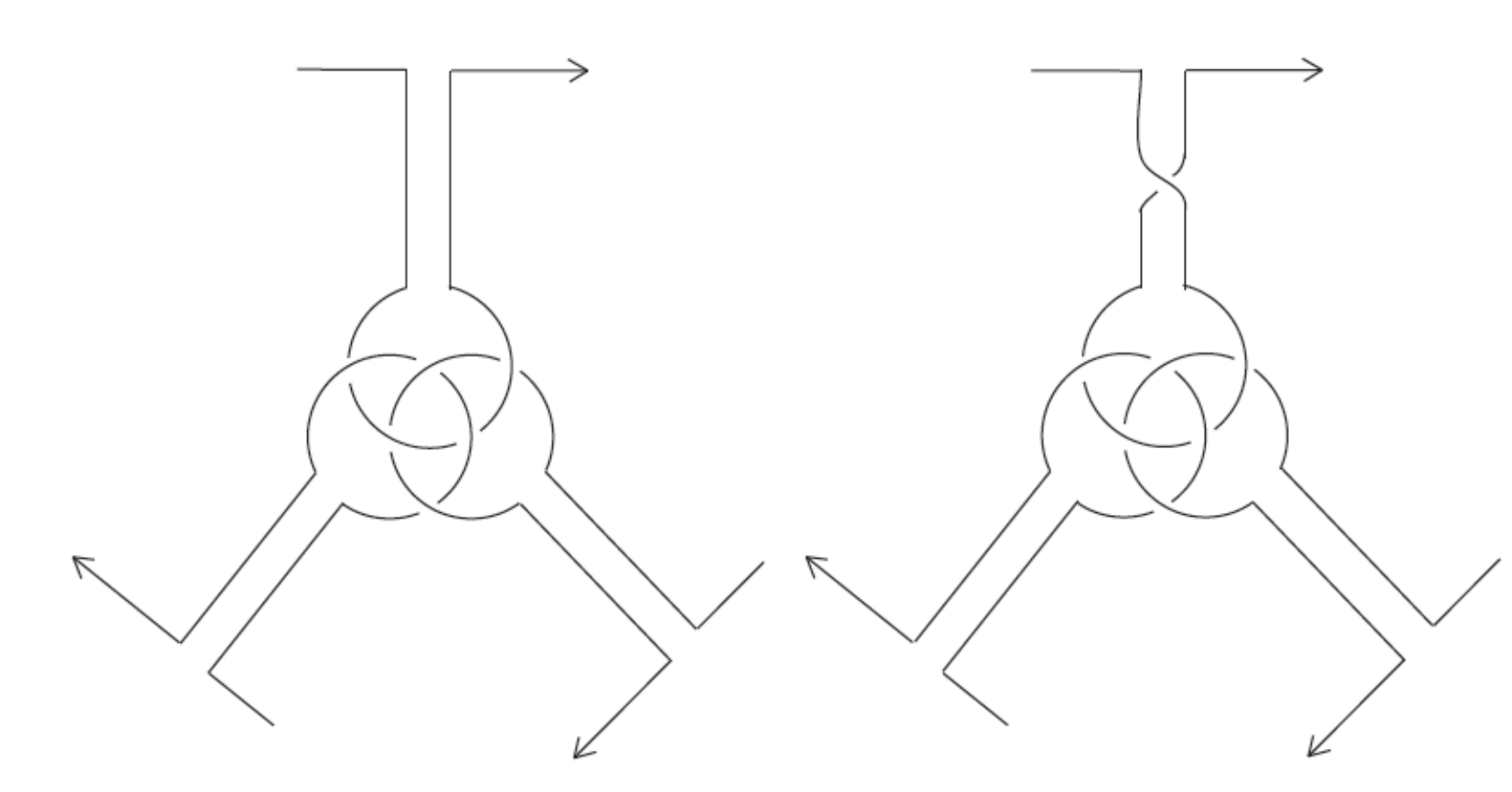}
\put(-300,70){(a)}
\put(-130,70){(b)}
\caption{Borromean chords of type $(ijk)$.}
\label{fig:typeijk}
\end{figure}

\item Each Borromean chord of $J$ of type $(ij)$ is contained in a 3-ball as illustrated in Figure \ref{fig:typeij} (c).

\begin{figure}[ht!]
\centering
\includegraphics[scale=.19]{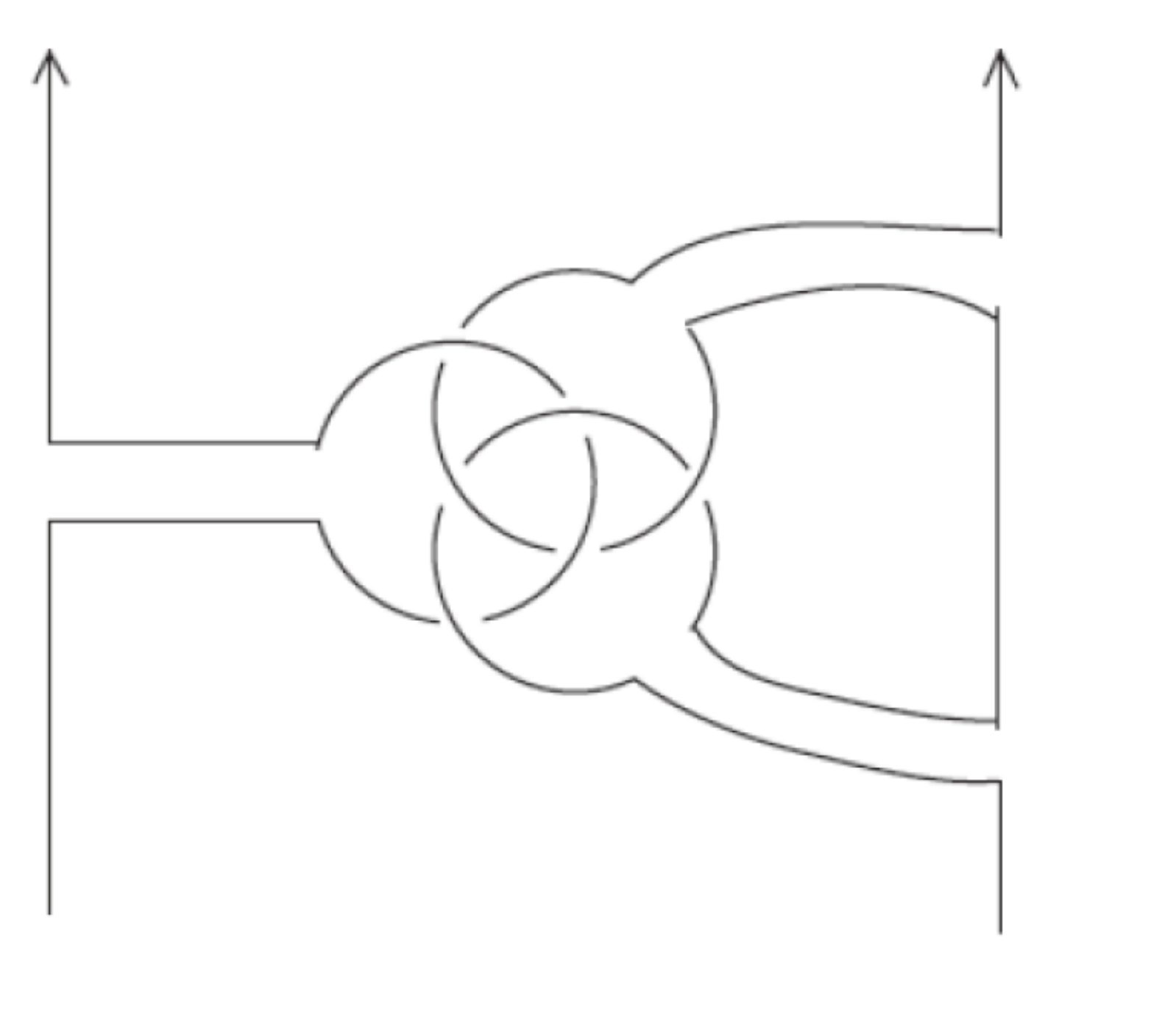}
\put(-180,60){(c)}
\put(-140,80){$K_j$}
\put(-38,60){$K_i$}
\caption{Borromean chords of type $(ij)$}
\label{fig:typeij}
\end{figure}

Furthermore, Taniyama and Yasuhara show in Lemma 2.5 that any two Borromean chords of type $(ij)$ cancel each other \cite{TY}. Therefore, for each $i < j \le m$, we may have at most one Borromean chord of type $(ij)$ as in Figure \ref{fig:typeij} (c).

\item Each Borromean chord of $J$ of type $(i)$ is contained in a 3-ball as illustrated in (e) or (f), and for each component, not both (e) and (f) occur. 

\begin{figure}[ht!]
\centering
\includegraphics[scale=.35]{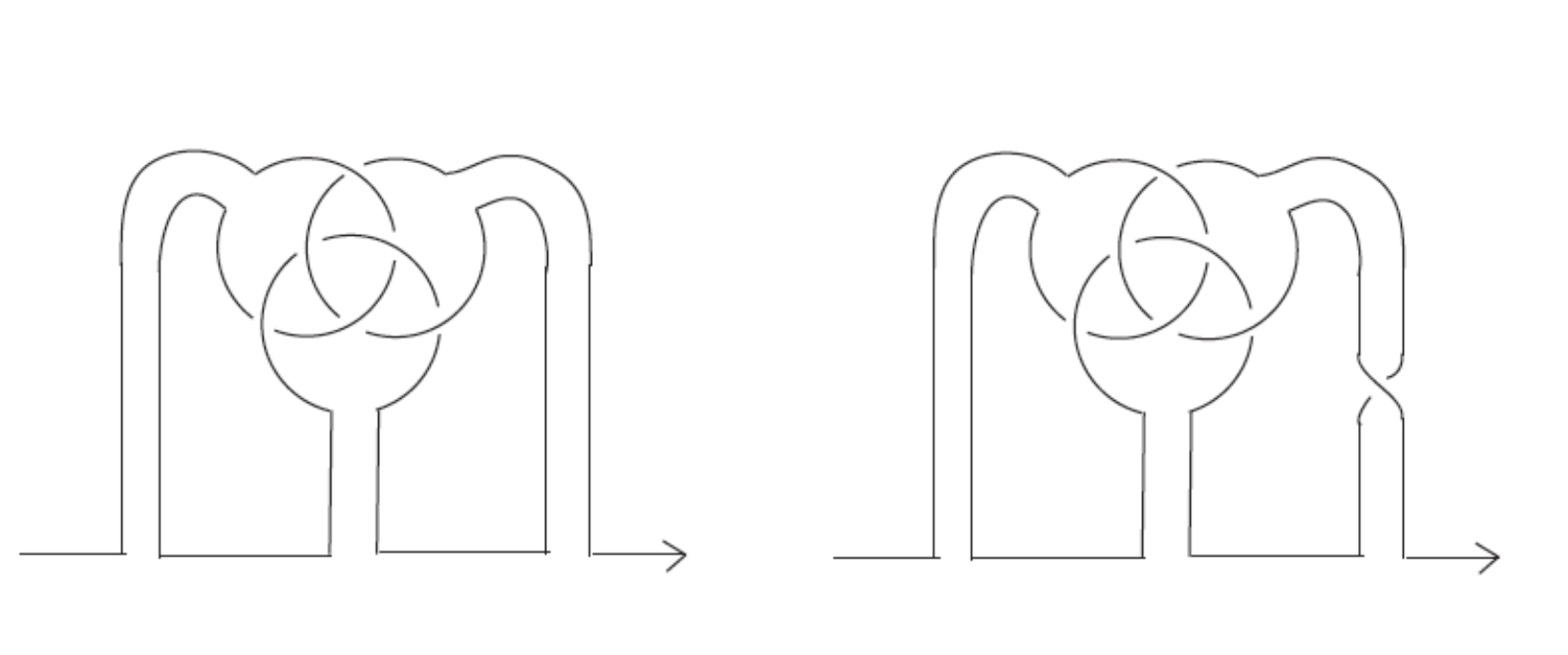}
\put(-280,50){(e)}
\put(-130,50){(f)}
\caption{Borromean chords of type $(i)$}
\label{fig:typei}
\end{figure}
\end{enumerate}

Therefore, let link $L$ be clasp-pass equivalent to a link $J$ of the above form, and let link $L'$ be clasp-pass equivalent to a link $J'$ of the above form. Using our assumptions on link $L$ and $L'$, we claim that $J $ and $J'$ are band-pass equivalent. 

From this construction, we see that all triple points contributing to $\bar{\mu}_L(ijk)$ come from the Borromean ring insertions; thus, $\bar{\mu}_L(ijk)$ is the signed number of Borromean chords of type $(ijk)$ in $J$, with sign $+1$ for chords as in figure (a) and sign $-1$ for chords as in figure (b). Given that $\bar{\mu}_L(ijk) = \bar{\mu}_{L'}(ijk)$, we know that link $J$ and link $J'$ have identical Borromean chords of type $(ijk)$.

Cited by Taniyama and Yasuhara \cite{TY} and due to a result of Hoste \cite{Hoste}, for a link $\link$, $a_3(K_i \cup K_j) \equiv 0$ $( \bmod \, 2 ) $ if and only if there are an even number of Borromean chords of type $(ij)$. Therefore, as we are assuming that $a_3(K_i \cup K_j) \equiv a_3(K_i' \cup K_j'
) $ $(\bmod 2)$, we know that for every choice of $i$ and $j$, either $J$ and $J'$ both have one Borromean chord of type $(ij)$ or they both have none.

Links $L$ and $J$ are clasp-pass equivalent, and links $L'$ and $J'$ are clasp-pass equivalent. The coefficient $a_2$ is preserved under clasp-pass equivalence \cite{TY}. In figure \ref{fig:typei}, the closure of (e) is a trefoil knot, and the closure of (f) is the figure eight knot. The invariants $a_2(\text{trefoil}) = 1$ and $a_2(\text{figure eight})=-1$. As the coefficient $a_2$ is additive under the connected sum of knots, the number of Borromean chords of type $(i)$ in link $J$ is given by $|a_2(K_i)|$, and the number of Borromean chords of type $(i)$ in link $J'$ is given by $|a_2(K_i')|$. 

Since $a_2(\text{trefoil}) \equiv a_2(\text{figure eight}) $ $(\bmod 2)$, (e) and (f) are band-pass equivalent \cite{Kauff}. Therefore, we may assume that all Borromean chords of both $J$ and $J'$ of type $(i)$ are as illustrated in figure \ref{fig:typei} (e). By enforcing the condition that $a_2(K_i) \equiv a_2(K_i') \bmod 2$, we are assuming that links $J$ and $J'$ have the same parity of Borromean chords of type $(i)$, all as in figure \ref{fig:typei} (e). In Lemma 2.5, Taniyama and Yasuhara show that any two such Borromean chords will cancel (see in particular, \cite{TY} figure 15). Therefore, for each $i$, we may assume that $J$ and $J'$ both have either one or zero Borromean chords of type $(i)$.

Therefore, we see that links $J$ and $J'$ are band-pass equivalent. We have that $L \sim_{\text{Clasp-Pass}} J \sim_{\text{Band-Pass}} J' \sim_{\text{Clasp-Pass}} L'$. Recalling that a clasp-pass move is also a band-pass move, we conclude that $L$ and $L'$ are band-pass equivalent.  
\end{proof}
Summarizing, we have shown that if $\link$ and $\linkprime$ are ordered, oriented, $m$-component links with vanishing pairwise linking numbers, and if $\Arf(K_i) = \Arf(K_i'), \bar{\mu}_L(ijk)=\bar{\mu}_{L'}(ijk)$, and $\bar{\mu}_L(iijj) \equiv \bar{\mu}_{L'}(iijj) ( \bmod{2} )$, then Proposition \ref{prop:Conway} tells us that $L$ and $L'$ satisfy the conditions of Proposition \ref{prop:TY}, which gives us that $L$ and $L'$ are band-pass equivalent. This completes the proofs of Lemma \ref{TandY} and Theorem \ref{thm:Main}.
\end{proof}

\section{Applications of the Main Theorem}

As a direct result of Theorem \ref{thm:Main}, the 0-solve equivalence class of a link $\link$ is characterized by the three algebraic invariants, $\Arf(K_i), \bar{\mu}_L(ijk),$ and $\bar{\mu}_L(iijj)$.  

\begin{corollary}\label{for:classify} The set $\LO^m = \el^m / \sim_0$ of concordance classes of $m$-component links with vanishing pairwise linking numbers up to 0-solve equivalence forms an abelian group under band sum such that, for each $m$, the natural map from the quotient $\frac{\F_{-0.5}^m}{\F_0^m}$ of string links with pairwise linking number 0 mod 0-solvable links to $\LO^m$ induces an isomorphism, and
$$ \LO^m \cong \frac{\F_{-0.5}^m}{\F_0^m} \cong \Z_2^m \oplus \Z^{m \choose 3} \oplus \Z_2^{m \choose 2}$$
\end{corollary}
 \begin{proof}

Consider the following as a diagram of sets, where $\F_{-0.5}^m$ is the set of concordance classes of string links with vanishing pairwise linking numbers, $\el^m / _\sim$ is the set of concordance classes of $m$-component links, $\pi$ is the standard quotient map, and $\wedge$ is the map taking a string link to its closure: 

\[
\begin{diagram}
\node{\F^m_{-0.5}} \arrow{e,t}{\wedge} \arrow{s,r}{\pi}
\node{\el^m/ _\sim} \arrow{s,r}{\bmod \sim_0} 
\\
\node{\frac{\F^m_{-0.5}}{\F^m_0}} \arrow{e,t}{\wedge}
\node{\LO^m} 
\end{diagram}
\] \\

\noindent We will show that two string links $A$ and $B$ are equivalent in $\frac{\F^m_{-0.5}}{\F^m_0}$ if and only if their closures $\hat{A}$ and $\hat{B}$ belong to the same 0-solve equivalence class in $\LO^m$. Suppose that $[A]=[B]$ in $\frac{\F^m_{-0.5}}{\F^m_0}$. Thus, $B = S \cdot A$ for some 0-solvable string link $S \in \F^m_0$ whose closure $\hat{S}$ is a 0-solvable link. We show that $\hat{A}$ and $\widehat{S \cdot A}$ are 0-solve equivalent links. First, the Arf invariant is additive under connected sum, and each component of $\hat{S}$ has Arf invariant zero, so the components of $\hat{A}$ and $\widehat{S \cdot A}$ have the same Arf invariants. Theorem 8 of \cite{Orr} shows that the first nonvanishing Milnor invariants are well-defined and additive under stacking. Thus, because all pairwise linking numbers of $\hat{A}$ and $\hat{B}$ vanish, $\bar{\mu}_{\widehat{S \cdot A}}(ijk)$ are well-defined, so $\bar{\mu}_{\hat{B}}(ijk) = \bar{\mu}_{\widehat{S \cdot A}}(ijk) = \bar{\mu}_{\hat{S}}(ijk) + \bar{\mu}_{\hat{A}}(ijk) = \bar{\mu}_{\hat{A}}(ijk)$. Finally, the values of $\bar{\mu}_{\widehat{S \cdot A}}(iijj)$ are determined by the 2-component $ij^{th}$ sublinks of $\widehat{S \cdot A}$. As two component links have no $\bar{\mu}(ijk)$, and the linking number of these two component sublinks vanish, the first nonvanishing Milnor's invariants, $\bar{\mu}(iijj)$ for each 2-component sublink of $\widehat{S\cdot A}$, are well-defined. Therefore, since $\bar{\mu}_{\hat{S}}(iijj) \equiv 0 \bmod 2$, we have that $\bar{\mu}_{\hat{B}}(iijj) = \bar{\mu}_{\widehat{S \cdot A}}(iijj) \equiv \bar{\mu}_{\hat{A}}(iijj) \bmod 2$. Thus, by Theorem \ref{thm:Main}, $\hat{B} = \widehat{S \cdot A}$ and $\hat{A}$ are 0-solve equivalent links. 

Similarly, we know that if links $\hat{A}$ and $\hat{B}$ are 0-solve equivalent, their components have equal Arf invariants, each choice of corresponding 3-component sublinks have equal $\bar{\mu}(123)$, and each choice of corresponding 2-component sublinks have $\bar{\mu}(1122)$ congruent mod 2. Using the additivity of these invariants as described above, we see that, for each component $K_i$ of the link $\widehat{AB^{-1}}$, $\Arf(K_i) = 0 $, for each 3-component sublink $J = K_i \cup K_j \cup K_k$ of $\widehat{AB^{-1}}$, $\bar{\mu}_{J}(123) = 0$, and for each 2-component sublink $N = K_i \cup K_j$ of $\widehat{AB^{-1}}$, $\bar{\mu}_N(1122)$ is even. Therefore, the link $\widehat{AB^{-1}}$ is 0-solvable, so the string link $AB^{-1}$ is 0-solvable. Thus, the string links $A$ and $B$ are equivalent in $\frac{\F^m_{-0.5}}{\F^m_0}$. This gives us a bijection between $\frac{\F^m_{-0.5}}{\F^m_0}$ and $\LO^m$.

As a result of Theorem \ref{thm:Main}, band sum is a well defined operation on $\LO^m$. With this operation, we see that the map $\wedge:\F^m_{-0.5}/\F_0^m \rightarrow \LO^m$ is a homomorphism, and thus an isomorphism.

Finally, we define a map $\psi : \F^m_{-0.5} \rightarrow \Z_2^m \oplus \Z^{m \choose 3} \oplus \Z_2^{m \choose 2}$ on an $m$-component string link $L$ with link components $K_1, \dots, K_m$ by  $$\psi(L) = (a_1, \dots, a_m, b_1, \dots, b_{m \choose 3},c_1,\dots,c_{m \choose 2}) \in \Z_2^m \oplus \Z^{m \choose 3} \oplus \Z_2^{m \choose 2}$$ where $a_i = \Arf (K_i)$, $b_k = \bar{\mu}_{J_k}(123)$ for $J_k$ the $k^{th}$ lexicographic three component sub link of $L$, and $c_n = \bar{\mu}_{T_n}(1122)$ for $T_n$ the $n^{th}$ lexicographic two component sub link of $L$.  Thus, $\ker(\psi) = \F^m_0$, and the additivity of invariants shows that $\psi$ is a homomorphism. 

Further, $\psi$ is surjective because each element in the codomain can be realized by stacking string links. The trefoil knot has Arf invariant 1. The Whitehead link has $\bar{\mu}_L(1122) = 1$. The Borromean rings have $\bar{\mu}_L(123)= \pm 1$. These links are pictured in figure \ref{fig:26}. Let $\vec{x} = (a_1, \dots, a_m, b_1, \dots, b_{m \choose 3},c_1,\dots,c_{m \choose 2}) \in \Z_2^m \oplus \Z^{m \choose 3} \oplus \Z_2^{m \choose 2}$. For the $m$-tuple $(a_1, \dots, a_m)$ with $a_r \in \Z_2$, form an $m$-component string link whose strands do not interact but where each strand is either unknotted (in the case where $a_r = 0$) or closes to the trefoil (in the case where $a_r = 1$). Then, for each element of $\{b_r\}_{r=1}^{m \choose 3}$ representing a choice of 3 components $i,j,k$, form a string link in which the $i,j,$ and $k^{th}$ strands form $b_r$ copies of the Borromean rings string link and all other strands are unknotted and unlinked. Finally, each element of $\{c_r\}_{r=1}^{m \choose 2}$ where $c_r =1$ represents a choice of 2 components $i,j$. We form a string link in which the $i^{th}$ and $j^{th}$ strands form the Whitehead string link and all other strands are unknotted and unlinked. We stack these string links together into a string link $S$, and we see that  $\psi(S)= \vec{x}$. Therefore, $\frac{\F^m_{-0.5}}{\F^m_0} \cong  \Z_2^m \oplus \Z^{m \choose 3} \oplus \Z_2^{m \choose 2}$.
 \end{proof}
 
 \begin{figure}[ht!]
 \centering 
\includegraphics[scale=.35]{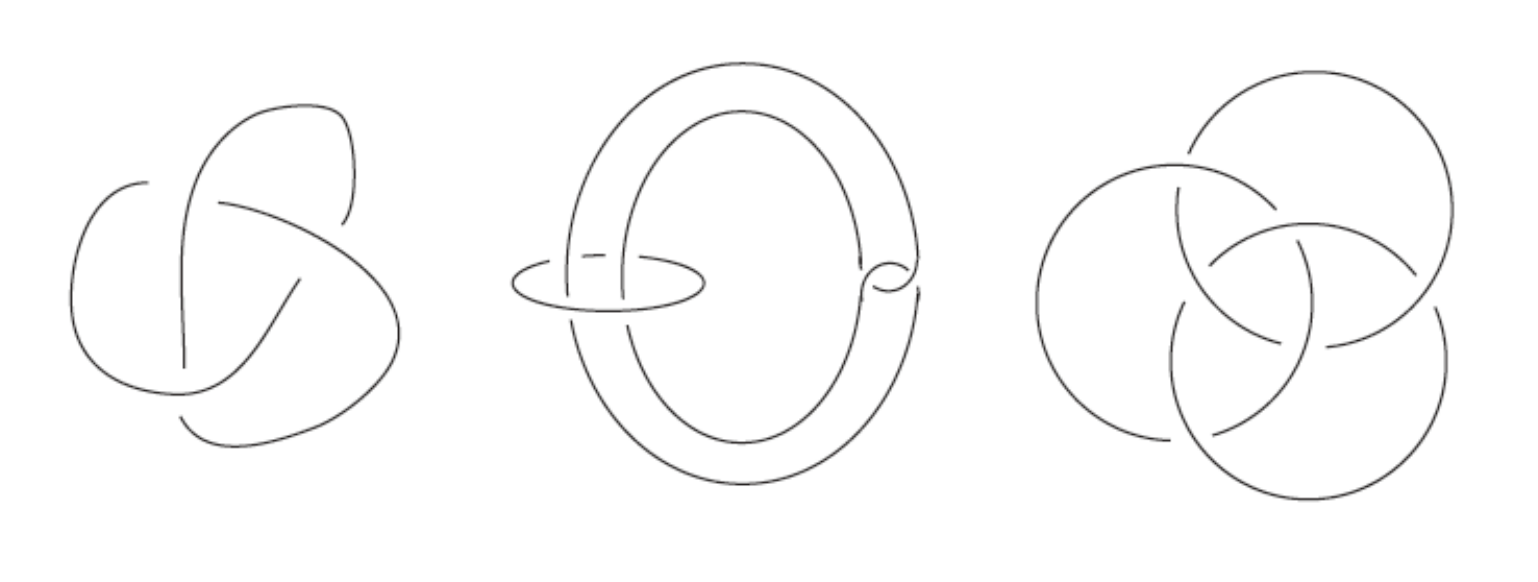}
\put(-250,0){$\Arf(K)=1$}
\put(-165,0){$\bar{\mu}_L(1122)=1$}
\put(-70,0){$\bar{\mu}_L(123)= \pm 1$}
\caption{Trefoil, Whitehead Link, and Borromean Rings}
\label{fig:26}
\end{figure}

 We can use this corollary to choose representatives for each 0-solve equivalence class of $m$-component links. For each element $(a_1, \dots, a_m, b_1, \dots, b_{m \choose 3},c_1,\dots,c_{m \choose 2}) \in \Z_2^m \oplus \Z^{m \choose 3} \oplus \Z_2^{m \choose 2}$, we choose a link representative $J = J_1 \cup \dots \cup J_m$ in the following way. The $i^{th}$ component of the representative will be either the unknot or the trefoil knot, according to if $a_i$ is 0 or 1. We order the triples of components $(J_i,J_j,J_k)$ lexicographically. For the $i^{th}$ triple $(J_i,J_j,J_k)$, link components, we insert $|c_i|$ Borromean rings with sign corresponding to the sign of $c_i$. We also order the pairs of components $(J_i,J_j)$ lexicographically. For the $i^{th}$ pair $(J_i,J_j)$, components $J_i$ and $J_j$ form an unlink or a Whitehead link according to if $c_i$ is 0 or 1. 

For example, when $m = 2$, we have an eight element group, $\C^2 / \F_0^2 \cong \Z_2^2 \oplus \Z_2,$ which are represented by the eight two-component links in figure \ref{fig:27}.

\begin{figure}[ht!]
\centering
\includegraphics[scale=.65]{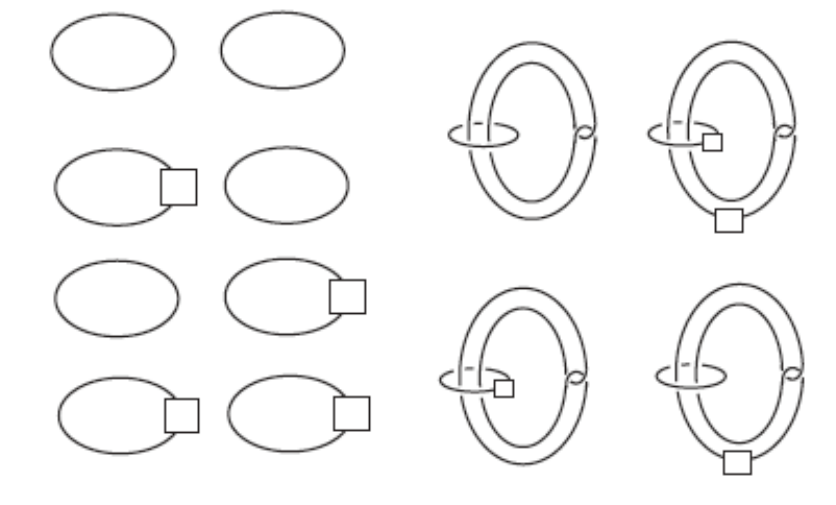}
\put(-280,140){(0,0,0)}
\put(-280,100){(1,0,0)}
\put(-280,60){(0,1,0)}
\put(-280,30){(1,1,0)}
\put(-208,96){\small T}
\put(-208,25){\small T}
\put(-154,25){\small T}
\put(-156,62){\small T}
\put(-110,150){(0,0,1)}
\put(-47,150){(1,1,1)}
\put(-110,0){(1,0,1)}
\put(-45,0){(0,1,1)}
\put(-36,87){\tiny T}
\put(-33,12){\tiny T}
\put(-41,111){\tiny T}
\put(-106,35){\tiny T}
\caption{2-component links up to 0-solve equivalence}
\label{fig:27}
\end{figure}

Figure \ref{fig:28} shows the 3-component link representing $(0,0,0,1,1,1,1) \in \Z_2^3 \oplus \Z \oplus \Z_2^3$.

 \begin{figure}[ht!]
 \centering
\includegraphics[scale=.3]{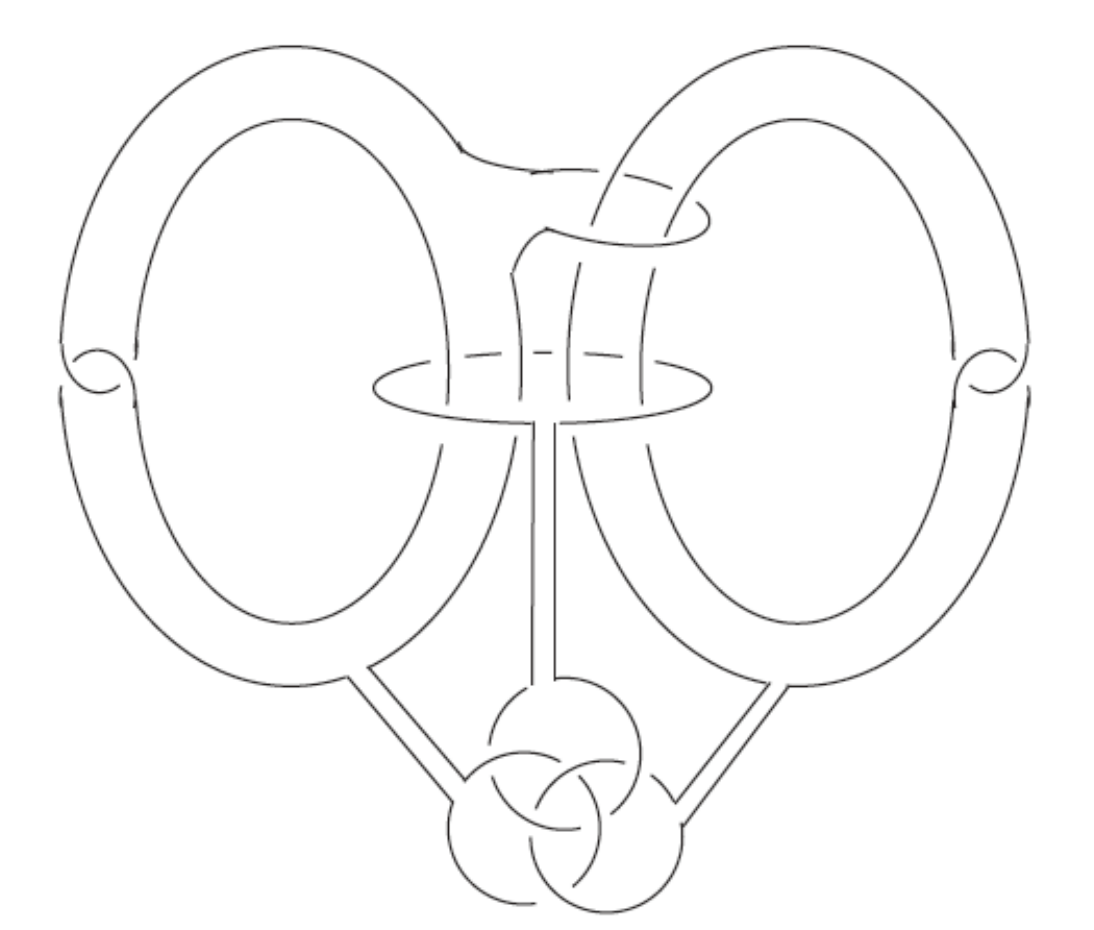}
\caption{Example: A 3-component 0-solve equivalence class}
\label{fig:28}
\end{figure}

Theorem \ref{thm:Main} has a direct application to the study of \emph{gropes} and \emph{Whitney towers}. Gropes and Whitney towers are geometric objects that play an important role in the study of 4-manifolds. Results due to Conant, Schneiderman, and Teichner will expand the statement of Theorem \ref{thm:Main}. A \emph{grope} is an oriented 2-complex created from joining oriented surfaces together in a prescribed way. Gropes have a natural complexity known as a \emph{class}. We give a definition from \cite{Rob}.

\begin{definition}
A \emph{grope} is a pair (2-complex, $S^1$) with a \emph{class} in $\N$. A class 1 grope is defined to be the pair $(S^1,S^1)$. A class 2 grope $(S, \bd S)$ is a compact oriented connected surface $S$ with a single boundary component. For $n >2$, a class $n$ grope is defined recursively. For a class 2 grope $(S, \bd S)$, let $\{\alpha_i,\beta_i\}_{i=1}^g$ be a symplectic basis for $H_1(S)$. For any $a_i,b_i \in \N$ such that $a_1+b_1 = n$ and $a_i + b_i \ge n$, a class $n$ grope is formed by attaching a class $a_i$ grope to $\alpha_i$ and a class $b_i$ grope to $\beta_i$. 
\end{definition}

\begin{figure}[ht!]
\centering
\includegraphics[scale=.35]{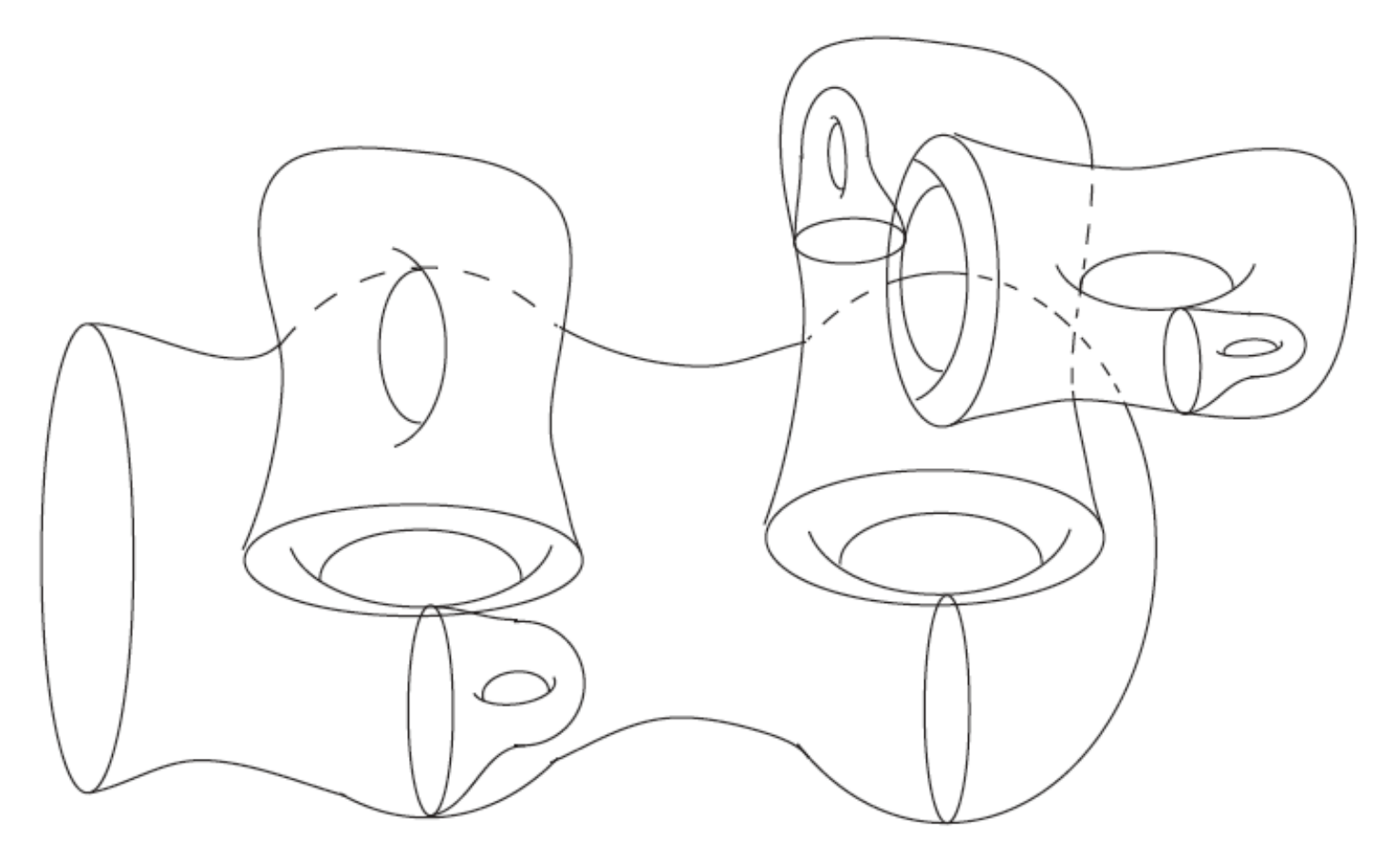}
\caption{A grope of class 4}
\label{fig:grope}
\end{figure}

We will consider gropes smoothly embedded in $\B^4$ with the boundary of the grope embedded in the $S^3 = \bd \B^4$. Therefore, the boundary of a grope will be a knot. We may then consider $m$ disjointly embedded gropes, which together are bounded by an $m$-component link.

One reason that 4-manifolds are not well understood is that the \emph{Whitney move} to eliminate intersections of immersed surfaces fails in four dimensions. To perform a Whitney move, we find an embedded \emph{Whitney disk} $W_{(I,J)}$ that pairs intersection points of two surface sheets $I$ and $J$ in a 4-manifold. We then change one surface, using $W_{(I,J)}$ as a guide. This is pictured in figure \ref{fig:Whit}. In 4-dimensions, we can eliminate the intersection between sheets $I$ and $J$, but we do so at the cost of possibly creating a new canceling pair of intersections, here between surface sheets $I$ and $K$. We may then look for another Whitney disk, $W_{(I,K)}$ to eliminate the new intersection. For more detailed definitions and background on this subject, see \cite{Freed}. 

\begin{figure}[ht!]
\centering
\includegraphics[scale=.45]{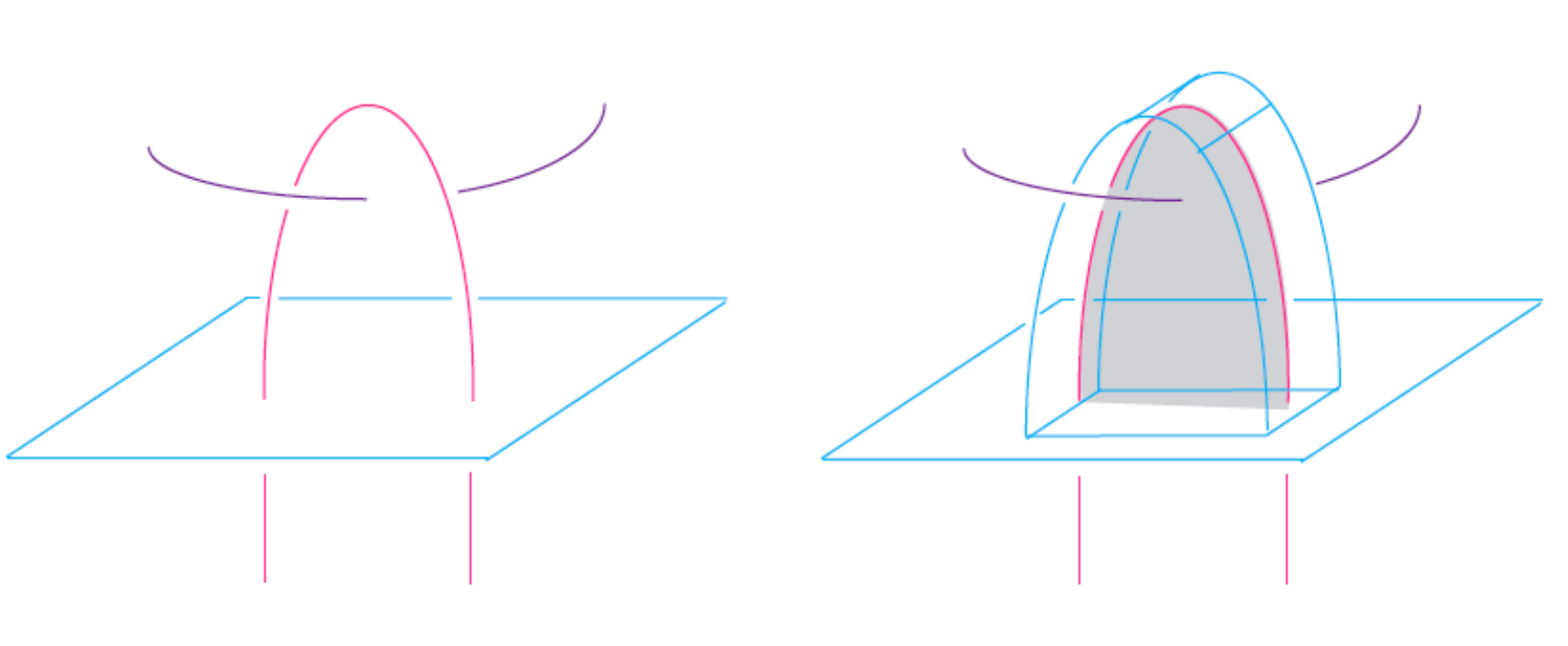}
\put(-97,75){$W_{(I,J)}$}
\put(-140,40){$I$}
\put(-90,20){$J$}
\put(-120,110){$K$}
\put(-200,55){$\longrightarrow$}
\caption{A Whitney disk and Whitney move}
\label{fig:Whit}
\end{figure}

We give the following definition from \cite{Rob}. An example of a Whitney tower is given in figure \ref{fig:Whitney}.

\begin{definition}

\begin{itemize}
\item A \emph{surface of order 0} in a 4-manifold $X$ is a properly immersed surface. A \emph{Whitney tower of order 0} in $X$ is a collection of order 0 surfaces.
\item The \emph{order of a transverse intersection point} between a surface of order $n$ and a surface of order $m$ is $n+m$.
\item A Whitney disk that pairs intersection points of order $n$ is said to be a \emph{Whitney disk of order} $(n+1)$.  
\item For $n \ge 0$, a \emph{Whitney tower of order $(n+1)$} is a Whitney tower $W$ of order $n$ together with Whitney disks pairing all order $n$ intersection points of $W$. The interiors of these top order disks are allowed to intersect each other as well as allowed to intersect lower order surfaces.
\end{itemize}
\end{definition}

\begin{figure}
\centering
\includegraphics[scale=.4]{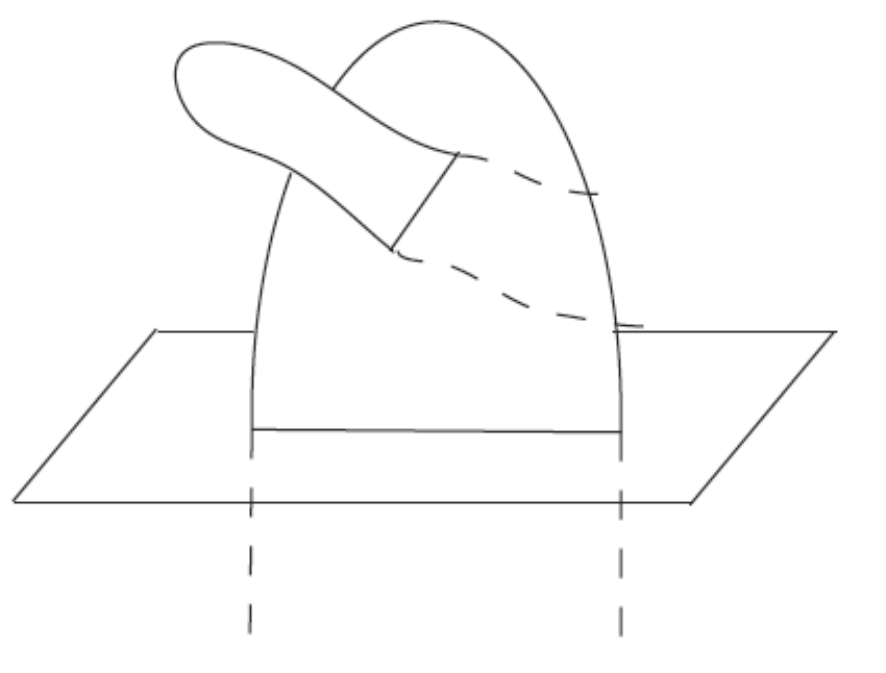}
\caption{A Whitney tower}
\label{fig:Whitney}
\end{figure}

We will consider Whitney towers in the 4-manifold $\B^4$. The sheets in our surfaces will all be disks, and we may then consider the boundaries of these disks in $\bd \B^4 = S^3$ to be links in $S^3$. Schneiderman relates curves bounding gropes to curves supporting Whitney towers in the following theorem \cite{Rob}.

\begin{thm}{\emph{[Schneiderman]\cite{Rob}}}

For any collection of embedded closed curves $\gamma_i$ in the boundary of $\B^4$,
the following are equivalent:
\begin{enumerate}
\item $\{\gamma_i\}$ bound disjoint properly embedded class $n$ gropes $g_i$ in $\B^4$.
\item $\{\gamma_i\}$ bound properly immersed 2-disks $\D_i$ admitting an order $(n-1)$ Whitney tower $W$ in $\B^4$. 
\end{enumerate}
\end{thm}

The following theorem, due to Conant, Schneiderman, and Teichner, relates Whitney towers to the algebraic link invariants used in Theorem \ref{thm:Main} \cite{Survey}.

\begin{thm}{\emph{[Conant-Schneiderman-Teichner]\cite{Survey}}}

A link $L$ bounds a Whitney tower $\mathcal{W}$ of order $n$ if and only if its Milnor invariants, Sato-Levine invariants $(\bmod 2)$, and Arf invariants vanish up to order $n$.
\end{thm}

We may then combine these results with Theorem \ref{thm:Main} to obtain the following result.

\begin{corollary}\label{cor:CST}
For an ordered, oriented, $m$-component link $L$, the following are equivalent.
\begin{enumerate}
\item $L$ is 0-solvable.
\item $L$ bounds disjoint, properly embedded gropes of class 2 in $\B^4$.
\item $L$ bounds properly immersed disks admitting an order 2 Whitney tower in $\B^4$.
\end{enumerate}
\end{corollary}

\section{Acknowledgements}

The author would like to S. Harvey for the many helpful conversations and encouragement, J. Conant for pointing out the connection to Whitney towers and gropes, and C. Davis and C. Otto for helpful feedback.

\bibliographystyle{plain}

\end{document}